\documentclass[11pt]{amsart}
\usepackage{amssymb}
\usepackage{tikz-cd}
\usepackage{graphicx}
\usepackage{subcaption}

\newfont{\msam}{msam10}
\newfont{\msbm}{msbm10}

\def\articletheorems{
\newtheorem{thm}{Theorem}[section]

\newtheorem{prop}[thm]{Proposition}
\newtheorem{ex}[thm]{Example}

}
\newcommand{\id}{\operatorname{id}}
\newcommand{\cl}{\operatorname{cl}}
\newcommand{\Cl}{\operatorname{cl}}
\newcommand{\Int}{\operatorname{int}}

\newcommand{\bd}{\operatorname{bd}}

\newcommand{\dom}{\operatorname{dom}}

\newcommand{\diam}{\operatorname{diam}}

\newcommand{\dist}{\mbox{$\operatorname{dist}$}}

\newcommand{\im}{\operatorname{im}}

\renewcommand{\emptyset}{\varnothing}

\newcommand{\Inv}{\operatorname{Inv}}

\def\mathobj#1{\mbox{$#1$}}

\def\RR{\mathobj{\mathbb{R}}}

\def\ZZ{\mathobj{\mathbb{Z}}}
\def\cA{\text{$\mathcal A$}}

\def\cG{\text{$\mathcal G$}}
\def\setof#1{\mbox{$\{\,#1\,\}$}}

\articletheorems
\newtheorem{lm}[thm]{Lemma}
\newtheorem{df}[thm]{Definition}
\newfont{\mbf}{msbm10 scaled 1100}
\newfont{\mmbf}{msbm10 scaled 800}
\def\R{\mbox{\mbf R}}

\def\Z{\mbox{\mbf Z}}

\def\N{\mbox{\mbf N}}

\def\z{\mbox{\mmbf Z}}

\def\for{\hspace{1ex} \mbox{ for }}

\usepackage[cp1250]{inputenc}

\newcommand{\Sol}{\operatorname{Sol}}

\newcommand{\mto}{\multimap}

\usepackage[cp1250]{inputenc}

\begin{document}
\title[]{Weak index pairs and the Conley index  for discrete multivalued dynamical systems. Part II: properties of the index}
\author{Bogdan Batko}
\date{}
\address{Bogdan Batko\\
Division of Computational Mathematics,
Faculty of Mathematics and Computer Science,
Jagiellonian University,
ul. St. Łojasiewicza 6,
30-348 Krak\'ow,
Poland
}
\email{bogdan.batko@ii.uj.edu.pl}
\date{}
\thanks{
   This research is partially supported by the Polish National Science Center under Ma\-estro Grant No. 2014/14/A/ST1/00453.
}

\begin{abstract}
Motivation to revisit the Conley index theory for discrete multivalued dynamical systems stems from the needs of broader real applications, in particular in sampled dynamics or in combinatorial dynamics. The 
new construction of the index in [B. Batko and M. Mrozek, {\em SIAM J. Applied Dynamical Systems}, 15(2016), pp. 1143-1162] based on weak index pairs, under the circumstances of the absence of index pairs caused by relaxing the isolation property, seems to be a promising step towards this direction. The present paper is a direct continuation of [B. Batko and M. Mrozek, {\em SIAM J. Applied Dynamical Systems}, 15(2016), pp. 1143-1162] and concerns properties of the index defined therin, namely Ważewski property, the additivity property, the homotopy (continuation) property and the commutativity property. We also present the construction of weak index pairs in an isolating block.
\end{abstract}
\maketitle
%-------------------------------------------------------------------
\section{Introduction}
The Conley index as a topological invariant defined for isolated invariant sets has become an important tool in the study of qualitative features of dynamical systems. The original construction of the index by Conley and his students in \cite{C78} concerned flows on locally compact metric spaces and was further generalized to arbitrary metric spaces (cf. \cite{R87,B91}), multivalued flows (cf. \cite{M90a}), discrete dynamical systems (cf. \cite{RS88,M90,DM93,FR00}), as well as discrete multivalued dynamical systems (cf. \cite{KM95}).

The interest in multivalued dynamics which admits multiple forward solutions, originated in the qualitative analysis of differential equations without uniqueness of solutions
and differential inclusions \cite{AC84}.
 
It turns out that multivalued dynamics can also be fruitfully applied while studying single valued dynamical systems, particularly in the rigorous numerical analysis of differential equations and iterates of maps. A rigorous numerical experiment, according to its nature, results in a multivalued map which covers the underlying single valued one.  
One can expect some of the qualitative features of the single valued dynamical system and its multivalued, reasonably tight, enclosure to be common and can be analyzed via topological invariants such as the Conley index. Such an approach was presented for the first time in \cite{MiMr95} and since then applied to many concrete problems.

The existing Conley index theory for discrete multivalued dynamical systems, originated by T. Kaczynski and M. Mrozek in 1995 (cf. \cite{KM95}), uses quite demanding definition of an isolating neigbourhood. Namely,  a compact set $N$ is an isolating neighborhood of a multivalued map $F$ if
\begin{equation}
\label{eq:mv-s-iso}
\dist(\Inv N,\bd N)> \max\{\diam F(x)\mid x\in N\},
\end{equation}
where $\Inv N$ stands for the invariant part of $F$ in $N$ (see Definition~\ref{def:inv}).
A slightly weaker but essentially similar definition is presented in \cite{S06}. Physical experiments show that the above assertion is restrictive and in practice finding an isolating neighborhood in the sense of (\ref{eq:mv-s-iso}) is difficult (cf. e.g. \cite{MiMrReSz99,MiMrReSz99b}). This limits the applicability of the theory. 

Therefore, in \cite{BM2016} the definition of an isolating neighborhood has been significantly generalized. Namely, instead of (\ref{eq:mv-s-iso}) it is required that 
\begin{equation}
\label{eq:mv-iso}
\Inv N\subset\Int N,
\end{equation}
i.e. the same condition as for single-valued maps. Such an approach, however, causes that index pairs, being the main tool in the construction of the Conley index, are no longer useful, because isolating neighborhoods in the sense of (\ref{eq:mv-iso}) do not guarantee their existence. Therefore, the new construction of the index was needed. The definition of the Conley index presented in \cite{BM2016} is based on weak index pairs.   

This paper is a direct continuation of \cite{BM2016} and is devoted to the properties of the Conley index defined therin. We discuss intrinsic properties of Conley type indices, namely Ważewski property, the additivity property, the homotopy (continuation) property and the commutativity property. Moreover, we present a simple construction of a weak index pair in an isolating block.  

The theory we develop may be useful among others in sampled dynamics, i.e. in the reconstruction problem of the qualitative features of an unknown dynamical system on the basis of a finite amount of experimental data only. For wider discussion concerning this issue we refer to \cite{BM2016}. Promising results in this direction are presented in \cite{MiMrReSz99,MiMrReSz99b} and also in a recent paper \cite{EdJaMr15}. In fact, in \cite{MiMrReSz99,MiMrReSz99b} the single valued Conley index theory is used, and the multivalued enclosure $F$ of the sampled dynamical system $f$ is necessary to aid the construction of index pairs only. Such an approach, however, requires the single valued generator to be the selector of the multivalued one which, in general, is not automatically guaranteed by the technique of its construction. We only can expect that $f$ lies nearby $F$. Enlarging the values of $F$ to ensure the existence of a continuous selector may cause overestimation and, as a consequence, prevent the isolation. Therefore, we advocate the qualitative features of the underlying dynamics to be inferred directly from the multivalued dynamical system $F$ constructed from the experimental data. To achieve this, the continuation property of the Conley index introduced in \cite{BM2016} is needed.   

Another potential application of our work concerns combinatorial dynamics. The recent results presented in \cite{KMW14} establish
formal ties between the classical dynamics and the combinatorial dynamics in the sense of Forman \cite{Fo98b}. The definition of the Conley index presented in \cite{BM2016} may enable extending these ties towards the Conley index theory. 

The organization of the paper is as follows.
Section \ref{sec:prel} presents preliminaries needed in the paper. 
In Section \ref{sec:C_ind} we recall the definition of an isolating neighborhood and the construction of the Conley index based on weak index pairs, following \cite{BM2016}. 
Section \ref{sec:ib} presents the definition of an isolating block for the  discrete multivalued dynamical system. An isolating block allows an easy construction of a weak index pair (cf. Theorem \ref{thm:wip_in_ib}) which may be convenient from the computational point of view. 
In Section \ref{sec:w_a_prop} Ważewski and the additivity properties are discussed. It turns out that, unlike in the single valued case, or even in the multivalued case for strongly isolated invariant sets, the union of two disjoint isolated invariant sets does not need to be an isolated invariant set (cf. Example \ref{ex:sum_not_io}). Even more, if the union of two disjoint isolated invariant sets is an isolated invariant set, its Conley index is not necessarily equal to the product of its components (cf. Example \ref{ex:not_add}). The reason is that the definition of an isolating neighborhood does not fully control the image of an isolated invariant set under $F$. As a consequence the sum of two isolated invariant sets may give birth to new trajectories connecting the summands. However, if the isolated invariant sets are sufficiently well separated from one another, in the sense that the image of any of them does not intersect the other, then the desired conclusion follows (cf. Theorem \ref{thm:add}). 
In Section \ref{sec:hom_prop} we discuss the continuation (homotopy) property of the Conley index (cf. Theorem \ref{thm:homotopy}). We also present its straightforward application to the leading example of \cite{BM2016} showing that the Conley index of the multivalued dynamical system constructed by sampling the dynamics coincides with the index of the underlying single valued dynamical system, as desired (cf. Example \ref{ex:hom1} and Example \ref{ex:hom2}).
In the last section we deal with the commutativity property. It seems that this property of the Conley index for discrete multivalued dynamical systems is undertaken here for the first time. 
%%%%%%%%%%%%%%%%%%%%%%%%%%%%%%%%%%%%%%%%%%%
\section{Preliminaries}\label{sec:prel}
Throughout the paper the sets of all integers, non-negative integers, non-positive integers and real numbers are denoted by $\Z$, $\Z^+$, $\Z^-$ and $\R$, respectively. By an interval in $\Z$ we mean a trace of a closed real interval in $\Z$.

By $\cl_XA$, $\Int _XA$ and $\bd _XA$ we denote
the closure, the interior and the boundary of a subset $A$ of a given topological space $X$. We drop the subscript $X$ in this notation
if the space is clear from the context.

Let $\mathcal{P}(Y)$ stand for the set of all subsets of a given topological space $Y$.
By a multivalued map $F:X\mto Y$ we mean a function $X\ni x\mapsto F(x)\in\mathcal{P}(Y)$. For given multivalued mapping $F:X\mto Y$ and $B\subset Y$ we define sets
$$
F^{-1}(B):=\{x\in X\mid  F(x)\cap B \neq\emptyset\},
$$
and
$$
F_{-1}(B):=\{x\in X\mid  F(x)\subset B\},
$$
called {\em large counter image} and {\em small counter image} of $B$ under $F$, respectively.
$F$ is said to be {\em upper semicontinuous} ({\em usc} for short) if the large counter image of any closed $B$ in $Y$ is closed in $X$ or, equivalently, if a small counterimage of any open $U$ in $Y$ is open in $X$. 
Given $A\subset X$ we define its {\em image} under $F$ to be 
$$
F(A):=\bigcup\{F(x)\mid x\in A\}.
$$ 
By $\im F$ we denote the {\em image} of $F$, i.e. the set $F(X)$.
Recall that any usc mapping with compact values has a closed graph and it sends compact sets into compact sets. If $F:X\mto Y$ is usc then its {\em effective domain}, i.e. the set $\dom (F):=\{x\in X\mid  F(x)\neq\emptyset\}$, is closed.

The {\em inverse} of a multivalued map $F:X\mto Y$ is a multivalued map $F^{-1}:Y\mto X$ defined by 
$$
x\in F^{-1}(y)\mbox{ if and only if }y\in F(x).
$$ 

For $F:X\mto Y$ and $G:Y\mto Z$ one defines the {\em composition} $G\circ F:X\mto Z$ by
$$
(G\circ F)(x):=\bigcup\{G(y)\mid y\in F(x)\}\for x\in X.
$$
Finally, if $F:X\mto X$ then by $F^k$, for $k\in\Z^+\setminus\{0\}$, we understand the  composition of $k$ copies of $F$, if $k$ is positive, or $-k$ copies of the inverse $F^{-1}$ of $F$, if $k$ is negative.
%------------------------------------
\begin{df}
{\rm (cf. \cite[Definition 2.1]{KM95}). An usc mapping $F:X\times\Z\mto X$ with compact values is called a {\em discrete multivalued dynamical system} ({\em dmds}) if the following conditions are satisfied:
\begin{itemize}
\item[(i)] for all $x\in X$, $F(x,0)=\{x\}$,
\item[(ii)] for all $n,m\in\Z$ with $nm\geq 0$ and all $x\in X$, $F(F(x,n),m)=F(x,n+m)$,
\item[(iii)] for all $x,y\in X$, $y\in F(x,-1)\Leftrightarrow x\in F(y,1)$.
\end{itemize}
}
\end{df}
Since $F^n(x):=F(x,n)$ coincides with a composition of $F^1:X\mto X$ or its inverse $(F^1)^{-1}$, we call $F^1$ the {\em generator} of the dmds $F$. For simplicity we denote the generator by $F$ and identify it with the dmds.

From now on we assume that $F$ is a given dmds.

\begin{df}
{\rm (cf. \cite[Definition 2.3]{KM95}). Let $I\subset\Z$ be an interval containing $0$. A single valued mapping $\sigma:I\to X$ is called a {\em solution for $F$ through $x\in X$} if $\sigma (0)=x$ and $\sigma (n+1)\in F(\sigma(n))$ for all $n,n+1\in I$.
}
\end{df}
\begin{df}
\label{def:inv}
{\rm Given $N\subset X$ we define the following sets
\begin{eqnarray*}
\Inv^+N&:=&\{x\in N\mid \exists\,\sigma:\Z^+\to N\mbox{ a solution  for }F\mbox{ through }x\},\\
\Inv^-N&:=&\{x\in N\mid \exists\,\sigma:\Z^-\to N\mbox{ a solution  for }F\mbox{ through }x\},\\
\Inv N&:=&\{x\in N\mid  \exists\,\sigma:\Z\to N\mbox{ a solution  for }F\mbox{ through }x\},
\end{eqnarray*}
called the {\em positive invariant part, negative invariant part} and the {\em invariant part of} $N$, respectively.
}
\end{df}
 Note that, by (i), $\Inv N=\Inv ^+N\cap \Inv ^-N$.

We will frequently consider pairs of topological spaces. For the sake of simplicity we will denote such pairs by single capital letters and then the first or the second element of the pair will be denoted by adding to the letter the subscript $1$ or $2$, respectively. In other words, if $P$ is a pair of spaces then $P=(P_1,P_2)$ where $P_1$, $P_2$ are topological spaces. Consequently, the rule extends to any relation $R$ between pairs $P$ and $Q$, i.e. any statement that pairs $P$ and $Q$ are in a relation $R$ will mean that $P_i$ is in a relation $R$ with $Q_i$ for $i=1,2$. According to our general assumption, whenever we say that $F$ is a map of pairs $P$ and $Q$ it means that $F$ maps $P_i$ into $Q_i$ for $i=1,2$.

Although most of the considerations in this paper are valid for locally compact topological spaces satisfying some separation axioms, usually locally compact Hausdorff spaces, for the sake of simplicity we make a general assumption that, until explicitely stated otherwise, by a space we mean a locally compact metrizable space. 

%%%%%%%%%%%%%%%%%%%%%%%%%%%%%%%%%%%%%%%%%%%%%%%%%
\section{Definition of the Conley index}\label{sec:C_ind}
In this section, following \cite{BM2016}, we summarize briefly the
important definitions related to isolating neighborhoods, isolated invariant sets, weak index pairs and the Conley index.

Assume $F:X\mto X$ is a given discrete multivalued dynamical system on a locally compact metrizable space. Recall that in \cite{KM95}
an isolating neighborhood $N$ in a locally compact metric space was defined as a compact set satisfying
\begin{equation}\label{inKM}
\dist (\Inv N,\bd N)>\max \{\diam F(x)\mid x\in N\}.
\end{equation}
Slightly relaxed but essentially similar condition
\begin{equation}\label{inKS}
\Inv N\cup F(\Inv N)\subset\Int N
\end{equation}
is used in \cite{S06}. In \cite{BM2016} the definition of an isolating neighborhood has been significantly generalized and, to avoid the misunderstanding, isolating neighborhoods in the sense of \cite{KM95} or \cite{S06} have been named {\em strongly isolating neighborhoods}. We follow this convention.  
\begin{df}{\rm (cf. \cite[Definition 4.1, Definition 4.2]{BM2016}).\label{def:iso_ne}
A compact subset $N\subset X$ is an {\em isolating neighborhood} for $F$ if $\Inv N\subset \Int N$. A compact set $S\subset X$ is said to be {\em invariant} with respect to $F$ if $S=\Inv S$. It is called an {\em isolated invariant set} if it admits an isolating neighborhood $N$ for $F$ such that $S=\Int N$. If in the above assertion $N$ is a strongly isolating neighborhood then we call $S$ a {\em strongly isolated invariant set}.
}
\end{df}
%-------------------------------------
The main tool in constructing the Conley index for flows, discrete dynamical systems as well as for discrete multivalued dynamical systems is an index pair. But, invariant sets isolated in the sense of Definition \ref{def:iso_ne} do not necessarily guarantee the existence of index pairs.
Therefore, in \cite{BM2016} weak index pairs are used. 

%-------------------------------------------------
We define an $F$-{\em boundary} of a given set $A\subset X$ by
$$
\bd _F(A):=\cl A\cap\cl (F(A)\setminus A).
$$
%Following \cite{M06} we recall the definition of a weak index pair.
\begin{df}\label{def:wip}
{\rm A pair $P=(P_1,P_2)$ of compact sets $P_2\subset P_1\subset N$ is called a {\em weak index pair} in $N$ if
\begin{itemize}
\item[(a)] $F(P_i)\cap N\subset P_i$ for $i\in\{1,2\}$,
\item[(b)] $\bd _FP_1\subset P_2$,
\item[(c)] $\Inv N\subset\Int (P_1\setminus P_2)$,
\item[(d)] $P_1\setminus P_2\subset\Int N$.
\end{itemize}
}
\end{df}
%-----------------------------------------
One can prove that any isolating neighborhood admits a weak index pair (cf. \cite[Theorem 4.12]{BM2016}). The construction of the Conley index is as follows.

We need the generator $F:X\mto X$ of the dmds, restricted to appropriate pairs of sets, to induce a homomorphism in cohomology. Therefore, we restrict ourselves to the class of maps {\em determined by a given morphism} (for the details see \cite{G83}, \cite{G76} or \cite{M90}). Let us recall that, in particular, any single-valued continuous map as well as any composition of acyclic maps (i.e. usc maps with compact acyclic values) belongs to this class.

For a weak index pair $P$ in an isolating neighborhood $N$ we let
\begin{equation}\label{eq:tp}
T(P):=T_N(P):=(P_1\cup (X\setminus \Int N), P_2\cup (X\setminus \Int N)).
\end{equation}
\begin{lm}{\rm (cf. \cite[Lemma 5.1]{BM2016})}\label{l1}
If $P$ is a weak index pair for $F$ in $N$ then
\begin{itemize}
\item[(i)] $F(P)\subset T(P)$,
\item[(ii)] the inclusion $i_{P,T(P)}:P\to T(P)$ induces an isomorphism in the Alexander-Spanier cohomology.
\end{itemize}
\end{lm}
Denote by $F_P:=F_{P,T(P)}$ the restriction of $F$ to the domain $P$ and codomain $T(P)$, and $i_P:=i_{P,T(P)}$. 
\begin{df}\label{df:ip}{\rm (cf. \cite[Definition 6.2]{BM2016})
The endomorphism $H^*(F_P)\circ H^*(i_P)^{-1}$ of $H^*(P)$ is called the {\em index map} associated with the index pair $P$ and denoted by $I_P$.
}
\end{df}
Several authors have used various concepts to construct the indices of Conley type: homotopy type (cf. \cite{RS88}), Leray functor (cf. \cite{M90}) and other normal functors (cf. \cite{M91,M94}), category of objects equipped with a morphism (cf. \cite{S95}) or shift equivalence (cf. \cite{FR00}). 

In our approach we apply the Leray functor $L$ to the pair $(H^*(P),I_P)$. Recall that the existence of a weak index pair $P$ in $N$ is guaranteed by \cite[Theorem 4.12]{BM2016}. We obtain a graded module over $\Z$ together with its endomorphism, called the {\em Leray reduction of the Alexander-Spanier cohomology} of $P$, which is independent of the choice of an isolating neighborhood $N$ for $S$ and of a weak index pair $P$ in $N$ (cf. \cite[Theorem 5.5]{BM2016}). Its common value is used to define the Conley index of $S$.
%-----------------------------------
\begin{df}{\rm (cf. \cite[Definition 6.3]{BM2016})\label{def:con}
The module $L(H^*(P),I_P)$ is called the {\em cohomological Conley index of $S$} and denoted by $C(S,F)$, or simply  by $C(S)$ if $F$ is clear from the context.}
\end{df}
%%%%%%%%%%%%%%%%%%%%%%%%%%%%%%%%%%%%%%%%%%%%%%%%
\section{Weak index pairs in isolating blocks}\label{sec:ib}
In this section we adapt the well known notion of  an {isolating block} to the case of discrete multivalued dynamical systems (see Definition \ref{df:ib}). As one can expect, any isolating block is an isolating neighborhood; hence it admits a weak index pair. The goal of this section is to show that an isolating block admits easier construction of a weak index pair than in general, i.e. in an arbitrary isolating neighborhood, which may be convenient in applications.

In this section $X$ is a locally compact Hausdorff space and $F:X\mto X$ is a discrete multivalued dynamical system.
%---------------------------------------
\begin{df}\label{df:ib}
{\rm We say that a compact set $N$ is an {\em isolating block} with respect to $F$, if 
$$
N\cap F(N)\cap F^{-1}(N)\subset \Int N,
$$
where $F^{-1}(N)$ stands for the large counterimage of $N$ under $F$.
}
\end{df}
%-------------------------------------
The following proposition is straightforward.
\begin{prop}
Any isolating block is an isolating neighborhood. The converse is not true.
\end{prop}    
\begin{proof}
The first statement is a straightforward consequence of the property $\Inv N\subset N\cap F(N)\cap F^{-1}(N)$. The second follows from Example \ref{ex:in_not_ib}.
\end{proof}
%-------------------------------------------
\begin{ex}\label{ex:in_not_ib}
{\rm
Consider $X=[0,6]$ and $F:X\mto X$ given by
\begin{equation}\label{eq:f_ib}
F(x):=\left\{
\begin{array}{rl}
\{0\}&\for x\in [0,1)\\
{[0,1]}&\for x=1\\
\{1\}&\for x\in (1,3)\\
{[1,5]}&\for x=3\\
{[3,5]}&\for x\in (3,4)\\
{[2,5]}&\for x=4\\
\{2\}&\for x\in (4,6].
\end{array}
\right.
\end{equation}
The graph of $F$ is presented in Figure \ref{fig:ib}. 
%--------------------------------
\begin{figure}[ht]
\begin{center}
\resizebox{80mm}{!}
{
\includegraphics{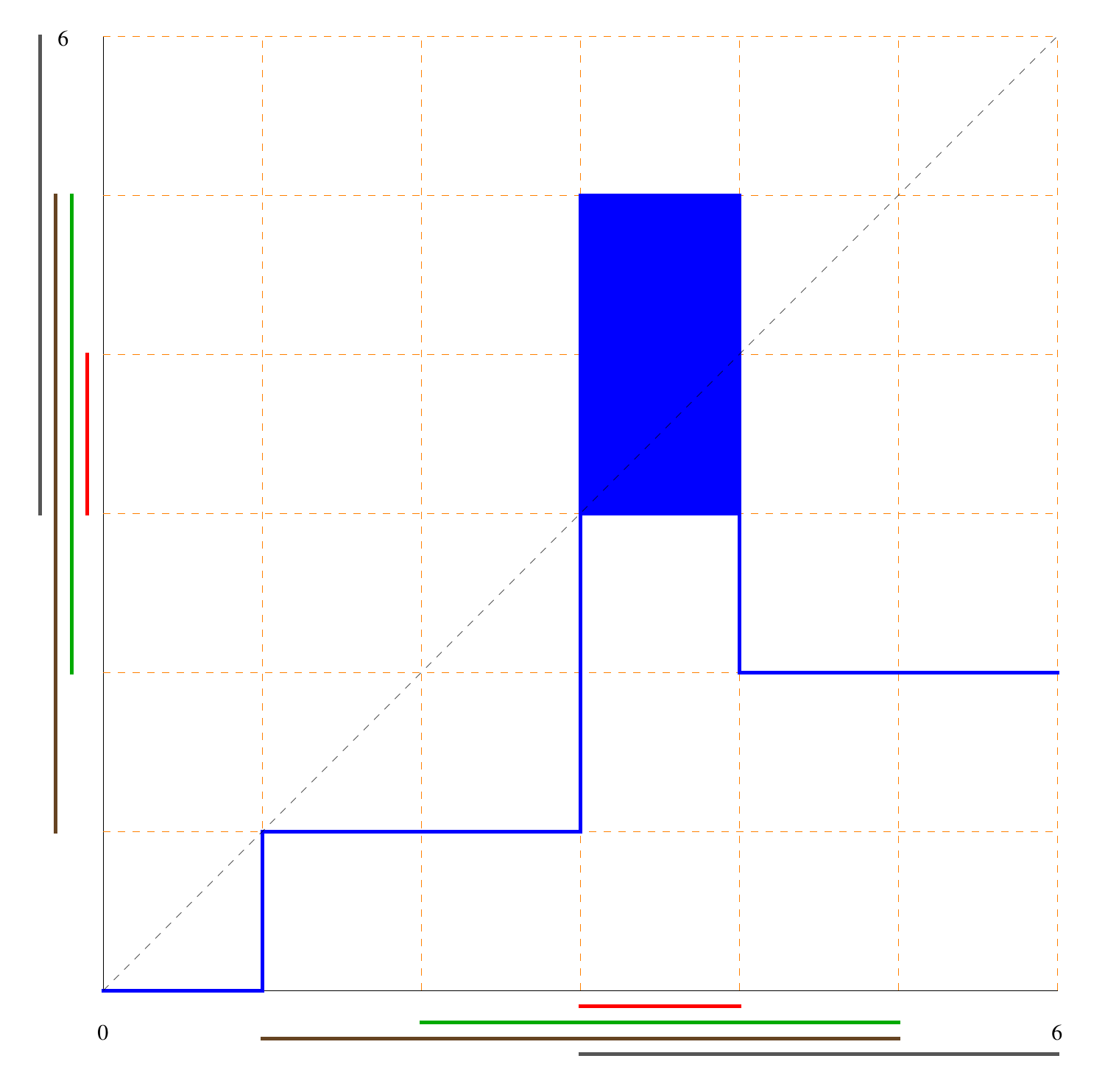}}
\caption{The graph of the map $F$ given by \eqref{eq:f_ib} is marked in blue.
Isolated invariant set $S$ and isolating neighborhood $N$ are marked in red and green, respectively.
The image $F(N)$ and the counterimage $F^{-1}(N)$, showing that $N$ is not an isolating block, %$F(N)\cap F^{-1}(N)\cap N \not\subset\Int N$, 
are marked in brown and gray, respectively.
}
\label{fig:ib}
\end{center}
\end{figure}
%--------------------------------
It is easy to see that $S=[3,4]$ is an isolated invariant set with respect to $F$ and $N=[2,5]$ isolates $S$, whereas $N$ is not an isolating block.
}
\end{ex}
%---------------------------------------------
\begin{thm}\label{thm:wip_in_ib}
Let $N$ be an isolating block with respect to $F$ and let $U$ be an open neighborhood of $F(N)\cap F^{-1}(N)\cap N$ with $\cl U\subset\Int N$. Then the sets $P_1:=(F(N)\cap N)\cup \cl U$ and $P_2:=F(N)\cap\bd N$ give raise to a weak index pair $P:=(P_1,P_2)$ in $N$.
%Let $N$ be an isolating block with respect to $F$. Then, the sets $P_1:=F(N)\cap N\cup \cl U$, where $U$ is an open neighborhood of $F(N)\cap F^{-1}(N)\cap N$ such that $\cl U\subset\Int N$, and $P_2:=F(N)\cap\bd N$, give raise to a weak index pair $P:=(P_1,P_2)$ in $N$.
\end{thm}
\begin{proof}
Clearly, $P_1$ and $P_2$ are compact, and $P_2\subset P_1\subset N$. 

The positive invariance of $P_1$ in $N$ is obvious. Thus, in order to prove property (a) of $P$, we shall verify that $F(P_2)\cap N\subset P_2$. Suppose the contrary and take $x\in F(P_2)\cap N\setminus P_2$. Clearly, $F(P_2)\cap N%=F(F(N)\cap\bd N)\cap N
\subset F(N)\cap N$, hence $x\in \Int N$. Since $x\in F(F(N)\cap\bd N)$, there exists $u\in F(N)\cap\bd N$ such that $x\in F(u)$. This yields $u\in \bd N\cap F(N)\cap F^{-1}(N)$, a contradiction. 

We shall verify condition (d). Observe that $P_1\setminus P_2=(F(N)\cap N\cup \cl U)\setminus (F(N)\cap \bd N)\subset F(N)\cap\Int N\cup\cl U$. This, along with $\cl U\subset\Int N$, yields $P_1\setminus P_2\subset \Int N$.    

Clearly, $U\subset P_1$ and $U\cap P_2=\emptyset$, as $P_2\subset\bd N$ and $U\subset\Int N$. Consequently, $U\subset P_1\setminus P_2$. Therefore, $\Inv N\subset F(N)\cap F^{-1}(N)\cap N\subset U\subset\Int (P_1\setminus P_2)$, which means that condition (c) holds.

It remains to verify property (b). Suppose the contrary and consider $x\in\bd _F(P_1)\setminus P_2$. Then $x\in P_1\setminus P_2$ and, by property (d), we have $x\in\Int N$. However, $x\in \bd_F(P_1)\subset \cl (F(P_1)\setminus P_1)$, thus we can take a sequence  $\{x_n\}\subset F(P_1)\setminus P_1$ convergent to $x$. Observe that $F(P_1)\setminus P_1\subset F(N)\setminus N$; hence $x_n\notin N$ for all $n\in\N$, which contradicts $x\in\Int N$.      
\qed\end{proof}
%--------------------------------
Observe that the assertion that $P_1$ contains a compact neighborhood of $F(N)\cap F^{-1}(N)\cap N$ is necessary.
\begin{ex}
{\rm Let $X=[0,7]$. Consider $F:X\mto X$ given by
\begin{equation}\label{eq:f_ib2}
F(x)=\left\{
\begin{array}{rl}
\{0\}&\for x\in [0,1)\\
{[0,3]}&\for x=1\\
\{3\}&\for x\in (1,3)\\
{[3,4]}&\for x\in [3,4]\\
\{3\}&\for x\in (4,6)\\
{[0,3]}&\for x=6\\
\{0\}&\for x\in (6,7].
\end{array}
\right.
\end{equation}
The graph of $F$ is presented in Figure \ref{fig:ib2}. 
%--------------------------------
\begin{figure}[ht]
\begin{center}
\resizebox{80mm}{!}
{
\includegraphics{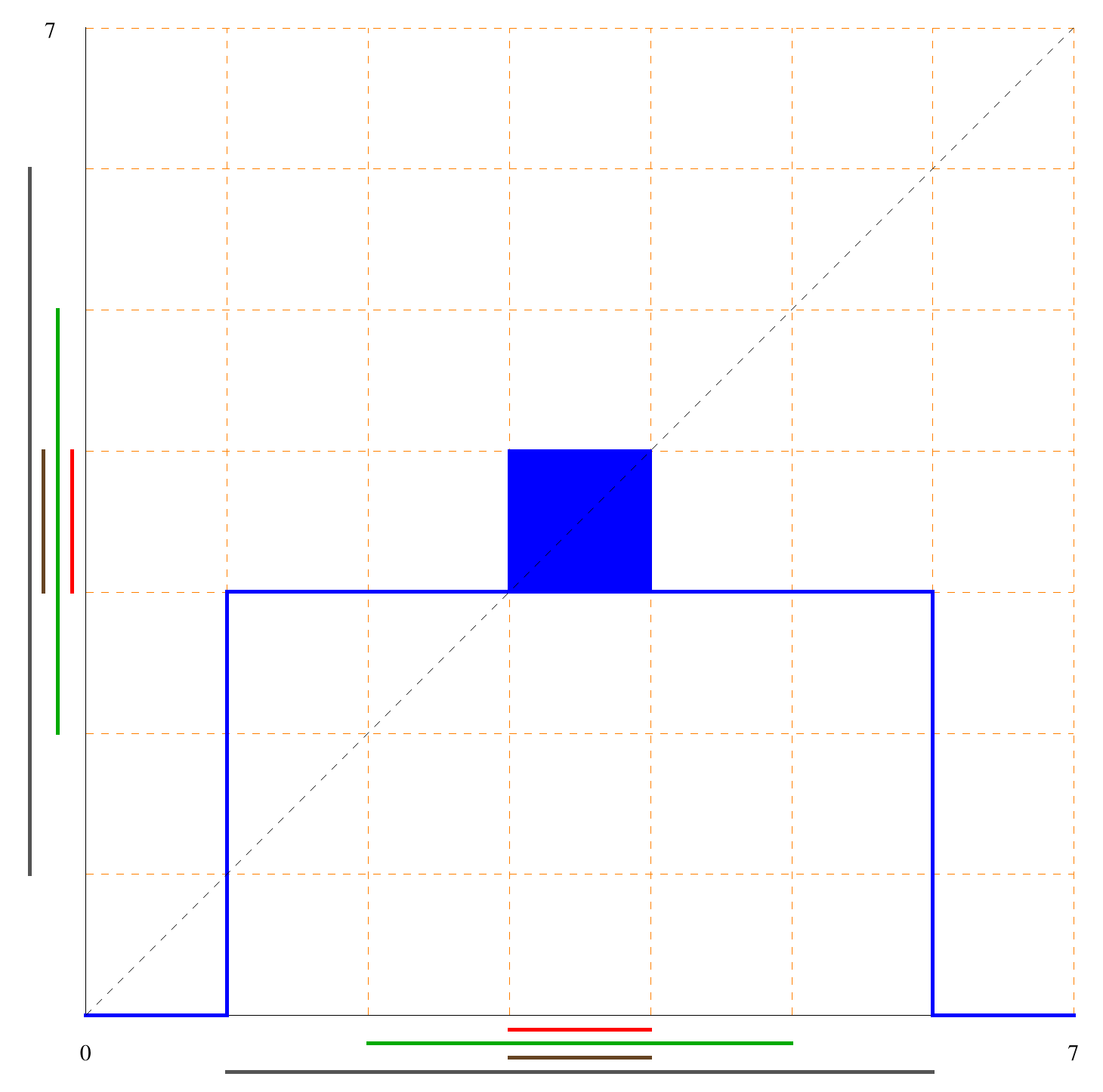}}
\caption{The graph of the map $F$ given by \eqref{eq:f_ib2} is marked in blue.
The isolating block $N$ and $\Inv N$ for $F$ are marked in green and red, respectively.
The image $F(N)$ and the counterimage $F^{-1}(N)$, showing that $\Inv N=F(N)\cap F^{-1}(N)\cap N$, are marked in brown and gray, respectively.
}
\label{fig:ib2}
\end{center}
\end{figure}
%--------------------------------
It is easy to see that $N=[2,5]$ is an isolating block with respect to $F$. Moreover, $\Inv N=F(N)\cap F^{-1}(N)\cap N$. If $P$ is a weak index pair with respect to $F$, then by property (c) of $P$, we have $F(N)\cap F^{-1}(N)\cap N\subset \Int P_1$. 
}
\end{ex}
To conclude this section let us emphasize that once an isolating block $N$ for a dmds given by a combinatorial multivalued map on a cubical grid is localized, the algorithmic construction of a weak index pair is straightforward. This is because a compact neighborhood of $F(N)\cap F^{-1}(N)\cap N$ required in Theorem \ref{thm:wip_in_ib} may be obtained by subdividing the grid, if necessary.
%%%%%%%%%%%%%%%%%%%%%%%%%%%%%%%%%%%%%%%%%%%%%%%%%
\section{Ważewski and additivity properties of the Conley index}\label{sec:w_a_prop}
Throughout this section we assume that $X$ is a locally compact metrizable space and $F:X\mto X$ is a discrete multivalued dynamical system.
 
\subsection{Ważewski property}
\begin{thm}
{\rm (}{\bf Ważewski property}{\rm )}
Let $S$ be an isolated invariant set with respect to $F$. If $C(S,F)\neq 0$, then $S\neq\emptyset$.
\end{thm}
\begin{proof}
The argument is by contradiction. Clearly $(\emptyset,\emptyset)$ is a weak index pair for $S=\emptyset$. Moreover, $H^*(\emptyset,\emptyset)=0$ and, by \cite[Proposition 4.6]{M90}, $L(0)=0$. Since the index is independent of the choice of an isolating neighborhood and a weak index pair (cf. \cite[Theorem 6.4]{BM2016}), we have $C(\emptyset,F)=0$.
\qed\end{proof}
%--------------------------------------------
\subsection{Additivity property}
Let us begin this section with an observation that, unlike in the case of single valued dynamical systems, or discrete multivalued dynamical systems and strongly isolated invariant sets, 
the union of two disjoint isolated invariant sets need not be an isolated invariant set.
\begin{ex}\label{ex:sum_not_io}
{\em Let 
$$
G(x):=\left\{
\begin{array}{rl}
\{0\}&\for x\in [0,1)\\
{[-3,0]}&\for x=1\\
\{-3\}&\for x\in (1,3)\\
{[-3,4]}&\for x=3\\
{[3,4]}&\for x\in (3,4)\\
{[3,6]}&\for x=4\\
\{6\}&\for x\in (4,7].
\end{array}
\right.
$$
Consider $X=[-7,7]$ and the dmds $F:X\mto X$ given by
\begin{equation}\label{eq:f_add1}
F(x):=\left\{\begin{array}{rl}
G(x)&\for x\in[0,7]\\
-G(-x)&\for x\in[-7,0).
\end{array}\right.
\end{equation}
The graph of $F$ is presented in Figure \ref{fig:add1}. 
%--------------------------------
\begin{figure}[ht]
\begin{center}
\resizebox{80mm}{!}
{
\includegraphics{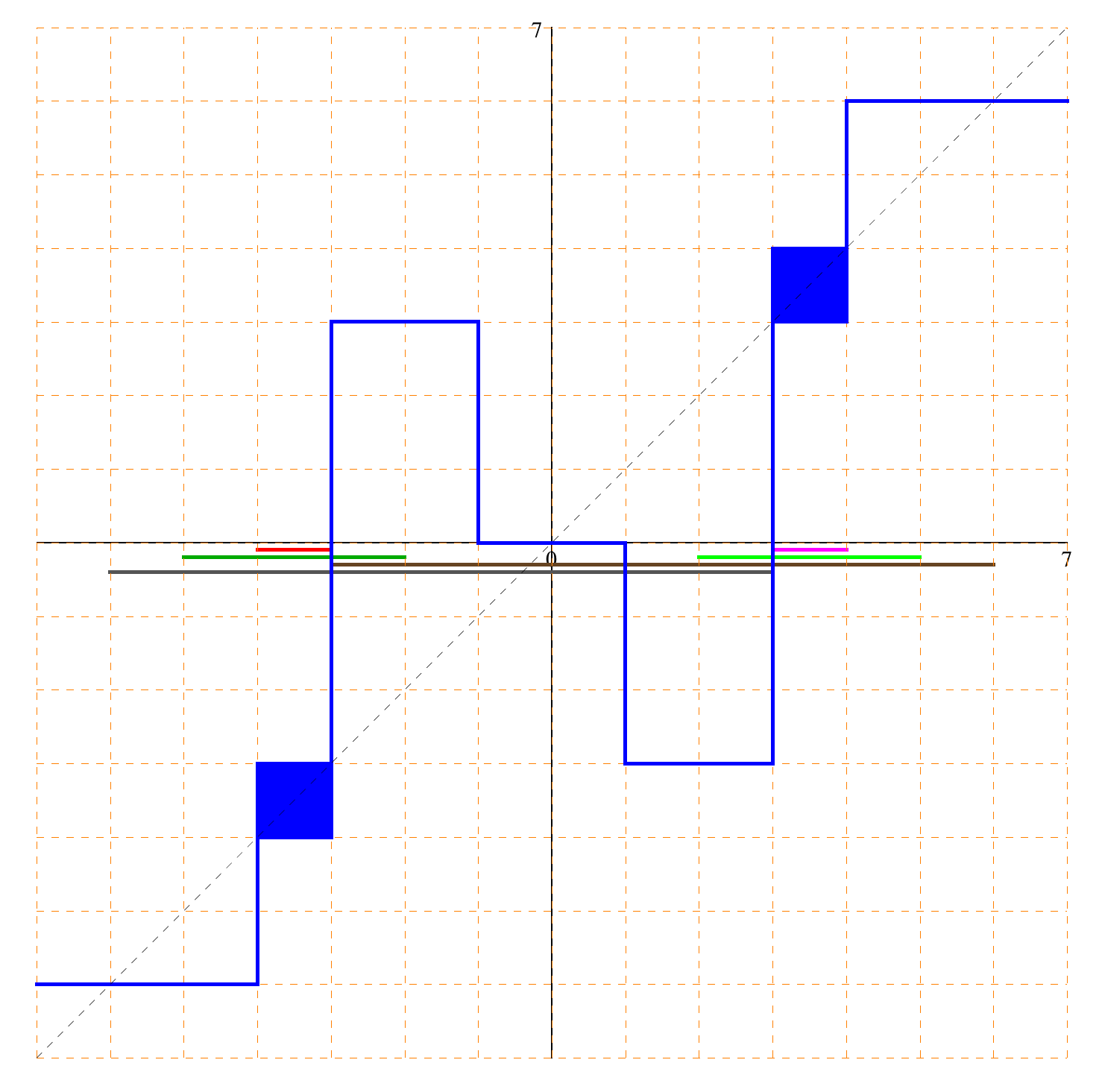}}
\caption{The graph of the map $F$ given by \eqref{eq:f_add1}, marked in blue.
Isolated invariant sets $S_1$ and $S_2$ are marked in red and pink, respectively. Isolating neighborhoods $N_1$ and $N_2$ for $S_1$ and $S_2$ are marked in dark green and green, respectively.
The images $F(S_1)$ and $F(S_2)$ are marked in gray and brown, respectively.
}
\label{fig:add1}
\end{center}
\end{figure}
%--------------------------------
It is easy to see that $S_1:=[-4,-3]$ and $S_2:=[3,4]$ are isolated invariant sets with respect to $F$, and $N_1:=[-5,-2]$, $N_2:=[2,5]$ isolate $S_1$ and $S_2$, respectively. Observe that any compact neighborhood $N$ of the union $S:=S_1\cup S_2$ has its invariant part significantly larger than $S$, which shows that $S$ is not an isolated invariant set. 
}
\end{ex}
Nevertheless, a suitably reformulated additivity property holds.
\begin{thm}\label{thm:add}
{\rm (}{\bf Additivity property}{\rm )}
Let an isolated invariant set $S$ for a discrete multivalued dynamical system $F$ on a locally compact metrizable space be a disjoint sum of two isolated invariant sets $S_1$ and $S_2$. Assume that 
\begin{equation}\label{disa}
F(S_i)\cap S_j=\emptyset\for i,j=1,2, i\neq j.
\end{equation}
Then $C(S,F)=C(S_1,F)\times C(S_2,F)$ {\rm (for the definition of the product $\times$ see \cite[Proposition 4.5]{M90})}.
\end{thm}

Before going to the technical details of the proof let us briefly describe its idea. First of all notice that we can not repeat the argument used for the proof of the additivity property in the single valued case. Namely, for the proof of \cite[Theorem 2.12]{M90} it is essential that for any disjoint isolated invariant sets $S_1$, $S_2$ one can choose isolating neighborhoods, and index pairs $P^1$, $P^2$ so that $F(P^1 _2)$ and $F(P^2 _2)$ are disjoint. This means that the pairs of spaces
$P^1 _i\cup F(P^1 _2)$ and $P^2 _i\cup F(P^2 _2)$, $i=1,2$, used for the construction of the index maps, are disjoint.
For convenience let us temporarily call such (index) maps {\em separated} from one another. One can observe that this is in contrast to the multivalued case, even if the underlying disjoint isolated invariant sets satisfy condition (\ref{disa}). 

In order to overcome this obstacle we shall use the construction of the extended discrete multivalued dynamical system presented in \cite[Section 7]{BM2016}, which we briefly recall.  

Let $N$ be an isolating neighborhood for a given dmds $F$. By \cite[Theorem 4.12]{BM2016} we can take a weak index pair $P$ in $N$, arbitrarily close to $\Inv N$, i.e. such that $P_1\setminus P_2\subset W$ and $W$ is an open neighborhood of $\Inv N$ with $\cl W\subset\Int N$. Then  
the sets $\cl (P_1\setminus P_2)$ and $X\setminus \Int N$ are disjoint; hence there exist compact and disjoint sets $C,D\subset P_1\cup (X\setminus \Int N)$ satisfying $\cl (P_1\setminus P_2)\subset C$ and $X\setminus \Int N\subset D$.
Now we use the Urysohn's lemma to choose a continuous function $\alpha :X\to [0,1]=:I$ such that $\alpha _{|_C}=0$ and $\alpha _{|_D}=1$.
Let
$$
X(P):=((P_1\setminus P_2)\times \{0\})\cup
((P_2\cup (X\setminus \Int N))\times I)\cup (X\times \{1\})
$$
with the Tichonov topology.
Consider $$\mu:X(P)\ni (x,t)\mapsto t+(1-t)\alpha (x)\in I$$ 
and define 
$$
F^P:X(P)\ni(x,t)\mapsto F(x)\times\{\mu(x,t)\}\subset X(P).
$$
Then $F^P$ is a well-defined, upper semicontinuous, acyclic valued map (cf. \cite[Proposition 7.2]{BM2016}).

For our needs we construct the space $X(P)$ for an appropriate, sufficiently small, isolating neighborhood $N$ of the union $S_1\cup S_2$, and a weak index pair $P$ in $N$. Then we consider adequate embedings ${S_i} '$ of the underlying isolated invariant sets $S_i$, $i=1,2$, into $X(P)$, one of them on level $0$ and the second on level $1$. The embedings are defined in such a way that they are isolated invariant sets with respect to $F^P$ and we have $C({S_i} ',F^P)=C(S_i,F)$, $C({S_1} ' \cup {S_2} ',F^P)=C(S_1\cup S_2,F)$. 
Clearly, this does not guarantee that the associated index maps are separated, regardless of the choice of weak index pairs ${P^i} '$
 for $S_i '$.  
 Recall that, in general, for the definition of index maps in the multivalued case we are forced to use pairs of spaces $T({P^1} ')$ and $T({P^2} ')$, which are not disjoint (see Definition \ref{df:ip}).
Although index maps themselves are not separated from one another, nice properties of the extended dynamical system $F^P$ facilitate the choice of weak index pairs which enable us to construct auxiliary maps that are isomorphic to the underlying index maps, and at the same time  separated, as desired. It turns out that this is sufficient for our purposes.\\

{\bf Proof of Theorem \ref{thm:add}:}
By (\ref{disa}) and the upper semicontinuity of $F$ we can take disjoint isolating neighborhoods $N_1 '$ and $N_2 '$ of $S_1$ and $S_2$, respectively, such that $F(N_i ')\cap N_j '=\emptyset$ for $i,j=1,2$, $i\neq j$. Let $N'$ be an isolating neighborhood of $S$. Put $N_i:= N_i '\cap N'$. Then $N_1$, $N_2$ and $N:=N_1\cup N_2$ are isolating neighborhoods of $S_1$, $S_2$ and $S$, respectively. Clearly,
\begin{equation}\label{fni}
F(N_i)\cap N_j=\emptyset\for i,j=1,2, i\neq j.
\end{equation}
By \cite[Theorem 4.12]{BM2016}, for arbitrarily small neighborhoods $W_i$ of $S_i$ there exist weak index pairs $P^i$ in $N_i$ with $P^i _1\setminus P^i _2\subset W_i\cap\Int N_i$, $i=1,2$. One can verify that condition (\ref{fni}) guaranties that $P:=P^1\cup P^2$ is a weak index pair in $N$.

Consider the space $X(P)$ and the dmds $F^P$, and the homeomorphism (onto its image)
$$
\iota_0:P_1\cup(X\setminus \Int N)\ni x\mapsto (x,0)\in\ X(P).
$$

By \cite[Proposition 7.5]{BM2016}, $\underline{S}:=\iota_0 (S)$ is an isolated invariant set with respect to $F^P$ and $\underline{N}:=\iota_0 ((P_1\cup(X\setminus \Int N))\cap N)$ is an isolating neighborhood of $\underline{S}$. Moreover, $\underline{P}:=\iota_0 (P)$ is a weak index pair in $\underline{N}$. 

Put $\underline{S_i}:=\iota_0 (S_i)$, $\underline{N}_i:=\iota_0 ((P^i _1\cup(X\setminus \Int N_i))\cap N_i)$ and $\underline{P^i}:=\iota_0 (P^i)$ for $i=1,2$. Since $N_1\cap N_2=\emptyset$, it is clear that 
$\underline{S}=\underline{S}_1\cup\underline{S}_2$, $\underline{N}=\underline{N}_1\cup\underline{N}_2$ and $\underline{P}=\underline{P}^1\cup\underline{P}^2$. 
Moreover, by (\ref{fni}) we have
\begin{equation}\label{fni2}
F^P(\underline{N}_i)\cap \underline{N}_j=\emptyset\for i,j=1,2, i\neq j.
\end{equation}
Thus, it follows that $\underline{N}_i$ is an isolating neighborhood of $\underline{S}_i$, and $\underline{P}^i$ is a weak index pair in $\underline{N}_i$, for $i=1,2$ (see Figure \ref{fig:add_p1}).    
%--------------------------------
\begin{figure}[ht]
\begin{center}
\resizebox{120mm}{!}
{
\includegraphics{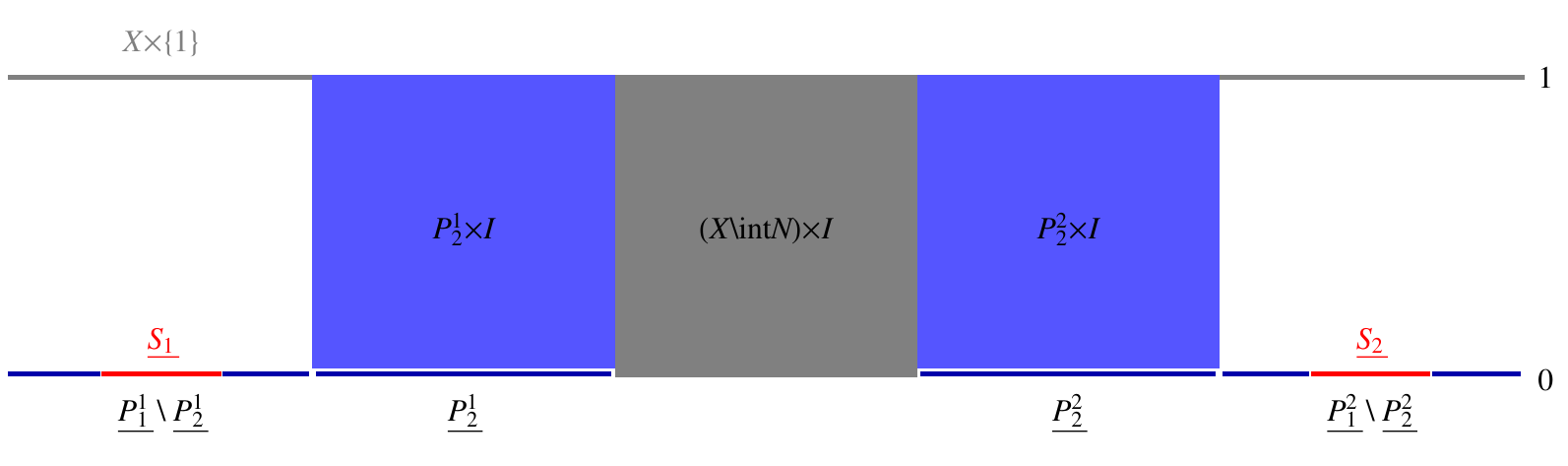}}
\caption{The extended space $X(P)$ and the embeddings $\underline{S}_i$ of the isolated invariant sets $S_i$. The isolated invariant sets $\underline{S}_1$ and $\underline{S}_2$ are indicated by red line segments. The components of the associated weak index pairs $\underline{P^1}$ and $\underline{P^2}$ are indicated by blue line segments.
}
\label{fig:add_p1}
\end{center}
\end{figure}
%--------------------------------
By \cite[Lemma 6.1(i)]{BM2016}, $F$ is a map of pairs $(P^1 _1,P^1 _2)$ and $(T_{N_1,1}(P^1),T_{N_1,2}(P^1))$. Set $(V_1,V_2):=(T_{\underline{N}_1,1}(\underline{P}^1)\setminus ((\Int N\setminus P)\times\{1\}),T_{\underline{N}_1,2}(\underline{P}^1)\setminus ((\Int N\setminus P)\times\{1\}))$. By \cite[Lemma 6.1(i)]{BM2016} and (\ref{fni}), $F^P$ maps $(\underline{P}^1 _1,\underline{P}^1 _2)$ into $(V_1,V_2)\subset (T_{\underline{N}_1,1}(\underline{P}^1),T_{\underline{N}_1,2}(\underline{P}^1))$.
We have the commutative diagram
$$
\begin{tikzcd}
(P^1 _1,P^1 _2) \ar{r}{F} & (T_{N_1,1}(P^1),T_{N_1,2}(P^1)) & (P^1 _1,P^1 _2) \ar{l}[swap]{j_1} \\[3ex]
(\underline{P}^1 _1,\underline{P}^1 _2) \ar{u}{p_1} \ar{r}{F^P}\ar{rd}[swap]{F^P} & (V_1,V_2)\ar{d}[swap]{j_4} \ar{u}{p_2} & (\underline{P}^1 _1,\underline{P}^1 _2) \ar{l}[swap]{j_2}\ar{u}{p_1}\ar{ld}{j_3}\\[3ex]
&(T_{\underline{N}_1,1}(\underline{P}^1),T_{\underline{N}_1,2}(\underline{P}^1))&
\end{tikzcd}
$$
in which $j_1, j_2, j_3$ and $j_4$ are inclusions and
$$
\begin{array}{l}
p_1:(\underline{P}^1 _1,\underline{P}^1 _2)\ni(x,0)\mapsto x\in (P^1 _1,P^1 _2),\\
p_2:(V_1,V_2)\ni (x,t)\mapsto x\in (T_{N_1,1}(P^1),T_{N_1,2}(P^1))
\end{array}
$$
are projections. By \cite[Lemma 6.1(ii)]{BM2016} 
inclusions $j_1$ and $j_3$ induce isomorphisms in cohomology. Observe that
$$
T_{N_1,1}(P^1)\setminus T_{N_1,2}(P^1)=(\underline{P}^1 _1\setminus \underline{P}^1 _2)\cap\Int \underline{N}_1
$$
and
$$
\begin{array}{rcl}
V_1\setminus V_2
&=&(T_{\underline{N}_1,1}(\underline{P}^1)\setminus ((\Int N\setminus P)\times\{1\}))\\[1ex]
&&\hspace{4cm}\setminus (T_{\underline{N}_1,2}(\underline{P}^1)\setminus ((\Int N\setminus P)\times\{1\}))
\\[1ex]
&=&(T_{\underline{N}_1,1}(\underline{P}^1)\setminus T_{\underline{N}_1,2}(\underline{P}^1))\setminus ((\Int N\setminus P)\times\{1\}))
\\[1ex]
&=&(\underline{P}^1 _1\setminus \underline{P}^1 _2)\cap\Int \underline{N}_1.
\end{array}
$$
Therefore,
inclusion $j_2$ induces an isomorphism, as an excision, and so does $j_4$. Clearly, projection $p_1$ induces an isomorphisms; hence so does $p_2$. Consequently, $I_{P^1}$ and $I_{\underline{P}^1}$ are conjugate; hence 
\begin{equation}\label{cs1}
C(S_1,F)=C(\underline{S_1},F^P).
\end{equation}
Now, consider a weak index pair $R$ in an isolating neighborhood $W:=\cl (P_1 ^1\setminus P_2 ^1)$ of $S_1$. It is easy to see that $\underline{W}:=\iota _0 (W)$ is an isolating neighborhood of $\underline{S_1}$ and $\underline{R}:=\iota _0 (R)$ is a weak index pair in $\underline{W}$ with respect to $F^P$. By the independence of the Conley index of the choice of an isolating neighborhood and of a weak index pair (cf. \cite[Theorem 6.4]{BM2016}), we have
\begin{equation}\label{eq:lcs_1}
C(\underline{S_1},F^P)=
L(H^*(\underline{R}),I_{\underline{R}}),
\end{equation}
where $L$ stands for the Leray functor.

By \cite[Lemma 6.1(i)]{BM2016}, $F^P(\underline{R})\subset T_{\underline{W}}(\underline{R})$. Moreover, using \cite[Corollary 7.3(i)]{BM2016} and (\ref{fni2}), we infer that 
\begin{equation}\label{eq:fp_R}
F^P (\underline{R})\subset \underline{T},
\end{equation}
where
$\underline{T}:=
\{(x,0)\in T_{\underline{W}}(\underline{R})\setminus\Int N_2\}$ (see Figure \ref{fig:add_r}). 
%--------------------------------
\begin{figure}[ht]
\begin{center}
\resizebox{120mm}{!}
{
\includegraphics{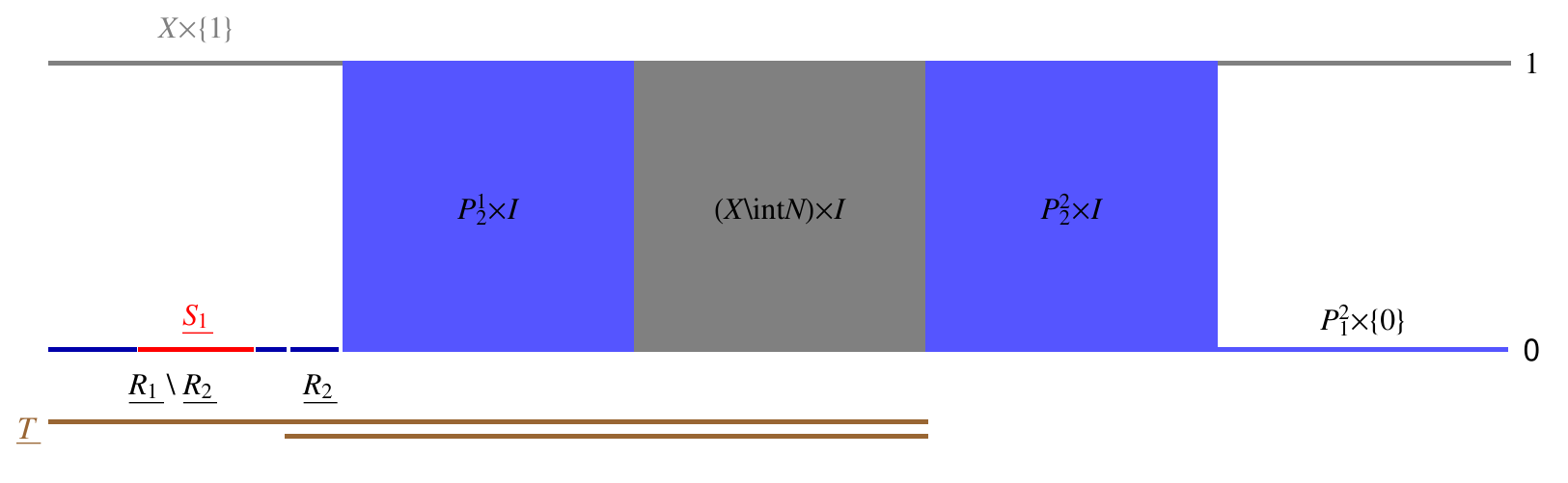}}
\caption{The extended space $X(P)$ and the mutual location of isolated invariant set $\underline{S}_1$, weak index pair $\underline{R}$ and spaces $\underline{T}$. $\underline{S}_1$ is marked in red. The weak index pair $\underline{R}$ is indicated by blue line segments. Components of the pair of spaces $\underline{T}$ lying on level $0$ are indicated by brown line segments.
}
\label{fig:add_r}
\end{center}
\end{figure}
%--------------------------------
We have the following commutative diagram
$$
\begin{tikzcd}
&(\underline{T}_1,\underline{T}_2)\ar{d}{i}&(\underline{R}_1,\underline{R}_2)\ar{dl}{i_{\underline{R}}}\ar{l}[swap]{i_{\underline{R},\underline{T}}} \\[3ex]
(\underline{R}_1,\underline{R}_2)\ar{ur}{F^P _{\underline{R},\underline{T}}} \ar{r}[swap]{F^P _{\underline{R}}}&(T_{\underline{W},1}(\underline{R}),T_{\underline{W},2}(\underline{R}))
\end{tikzcd}
$$
in which by $i$ (with a subscript) we denote inclusions and by $F^P$ with subscripts, the contractions of $F^P$ to appropriate pairs of pairs. By \cite[Lemma 6.1(i)]{BM2016}, inclusion $i_{\underline{R}}$ induces an isomorphism in the Alexander-Spanier cohomology. 
Observe that
$$
\begin{array}{rcl}
T_{\underline{W},1}(\underline{R})\setminus T_{\underline{W},2}(\underline{R})
&=&(\underline{R}_1\cup (X(P)\setminus\Int \underline{W}))\setminus (\underline{R}_2\cup (X(P)\setminus\Int \underline{W}))\\[1ex]
&=&(\underline{R}_1\setminus \underline{R}_2)\cap \Int \underline{W}
\end{array}
$$
and
$$
\begin{array}{rcl}
\underline{T}_1\setminus\underline{T}_2
&=&\{(x,0)\in T_{\underline{W},1}(\underline{R})\setminus\Int N_2\}\setminus \{(x,0)\in T_{\underline{W},2}(\underline{R})\setminus\Int N_2\}\\[1ex]
&=&\{(x,0)\in (T_{\underline{W},1}(\underline{R})\setminus T_{\underline{W},2}(\underline{R}))\setminus\Int N_2\}\\[1ex]
&=&\{(x,0)\in ((\underline{R}_1\setminus \underline{R}_2)\cap \Int \underline{W})\setminus\Int N_2\}\\[1ex]
&=&\{(x,0)\in (\underline{R}_1\setminus \underline{R}_2)\cap \Int \underline{W}\}\\[1ex]
&=&(\underline{R}_1\setminus \underline{R}_2)\cap \Int \underline{W}.
\end{array}
$$
Thus, by the strong excision property, inclusion $i_{\underline{R},\underline{T}}$ induces an isomorphism in the Aleksander-Spanier cohomology, and so does $i$. Therefore, we have a well-defined endomorphism $H^*(F^P _{\underline{R},\underline{T}})\circ H^*(i_{\underline{R},\underline{T}})^{-1}$ of $H^*(\underline{R})$.  
Moreover, by the commutativity of the above diagram
\begin{equation}\label{eq:ir_ir}
I_{\underline{R}}=H^*(F^P _{\underline{R},\underline{T}})\circ H^*(i_{\underline{R},\underline{T}})^{-1}.
\end{equation}

Consider the homeomorphism (onto its image)
$$
\iota _1:X\ni x\mapsto (x,1)\in X(P)
$$
and define $\overline{S}_i:=\iota_1 (S_i)$, $\overline{N}_i:=\iota_1 (N_i)$, and $\overline{P^i}:=\iota_1 (P^i)$, for $i=1,2$.
One can verify, using 
\cite[Proposition 7.1]{BM2016}, that $\overline{S}_i$ is an isolated invariant set with respect to $F^P$, $\overline{N}_i$ is an isolating neighborhood, and $\overline{P^i}$ is a weak index pair in $\overline{N}_i$. 

By \cite[Lemma 6.1(i)]{BM2016}, $F$ is a map of pairs $(P^2 _1,P^2 _2)$ and \linebreak $(T_{N_2,1}(P^2),T_{N_2,2}(P^2))$. Moreover, by \cite[Lemma 6.1(i)]{BM2016}, (\ref{fni}) and \cite[Proposition 7.1.3]{BM2016}, we have
\begin{equation}\label{eq:fp_P2}
F^P(\overline{P^2})\subset \overline{T},
\end{equation}
where
$(\overline{T}_1,\overline{T}_2):=(\iota_1(P^2 _1\cup(X\setminus \Int N)),\iota_1(P^2 _2\cup(X\setminus \Int N)))$ (see Figure \ref{fig:add_p2}).
%--------------------------------
\begin{figure}[ht]
\begin{center}
\resizebox{120mm}{!}
{
\includegraphics{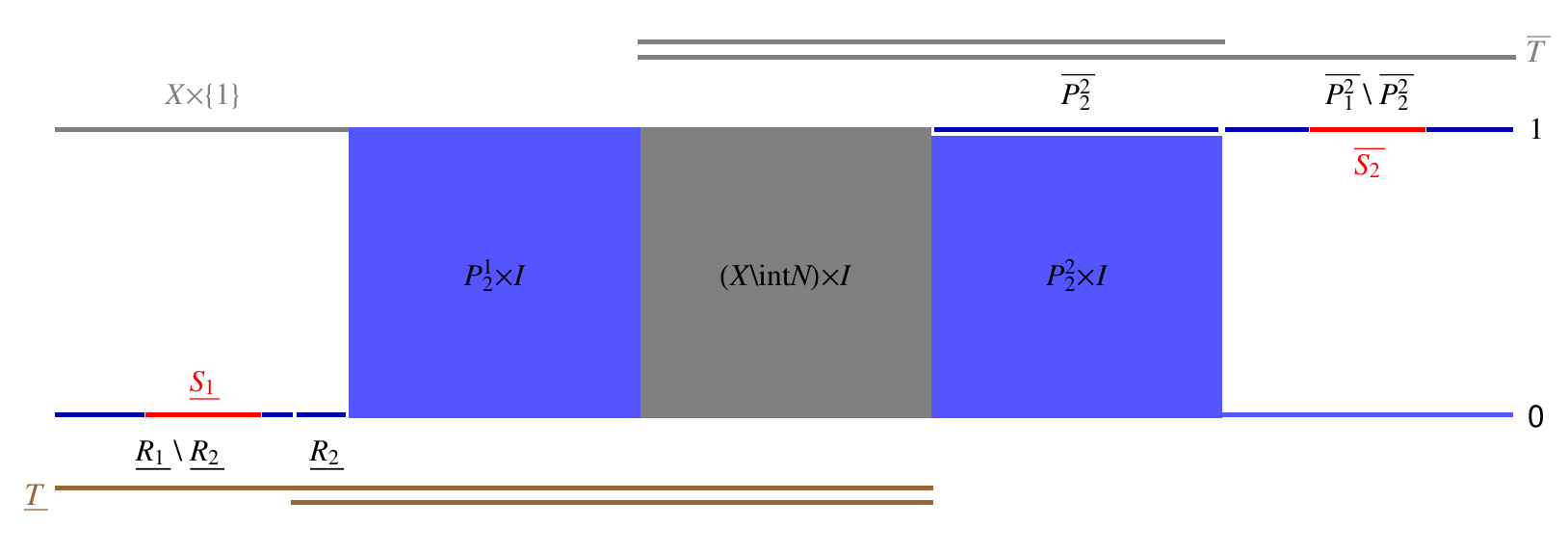}}
\caption{The extended space $X(P)$ and the mutual location of isolated invariant sets 
$\underline{S_1}$ and $\overline{S_2}$, weak index pairs $\underline{R}$ and $\overline{P^2}$, and pairs of spaces $\underline{T}$ and $\overline{T}$. The sets $\underline{S_1}$ and $\overline{S_2}$ on level $0$ and level $1$, respectively, are marked in red. The associated weak index pairs $\underline{R}$ and 
$\overline{P^2}$ are indicated by blue line segments. Brown and gray line segments indicate location of components of the pairs of spaces $\underline{T}$ on level $0$ and $\overline{T}$ on level $1$, respectively.
}
\label{fig:add_p2}
\end{center}
\end{figure}
%--------------------------------

Clearly, $(\overline{T}_1,\overline{T}_2)\subset (T_{\overline{N}_2,1}(\overline{P}^2),T_{\overline{N}_2,2}(\overline{P}^2))$. We have the following commutative diagram
$$
\begin{tikzcd}
(P^2 _1,P^2 _2) \ar{r}{F} & (T_{N_2,1}(P^2),T_{N_2,2}(P^2)) & (P^2 _1,P^2 _2) \ar{l}[swap]{j_5} \\[3ex]
(\overline{P}^2 _1,\overline{P}^2 _2) \ar{u}{p_3} \ar{r}{F^P _{\overline{P^2}}}\ar{rd}[swap]{F^P_{\overline{P^2},\overline{T}}} & (T_{\overline{N}_2,1}(\overline{P}^2),T_{\overline{N}_2,2}(\overline{P}^2)) \ar{u}{p_4} & (\overline{P}^2 _1,\overline{P}^2 _2) \ar{l}[swap]{j_6}\ar{u}{p_3}\ar{ld}{i_{{\overline{P^2},\overline{T}}}}\\[3ex]
&(\overline{T}_1,\overline{T}_2)\ar{u}{j_7}&
\end{tikzcd}
$$
in which $j_5, j_6, j_7$ and $i_{\overline{P^2},\overline{T}}$ are inclusions and
$$
\begin{array}{l}
p_3:(\overline{P}^2 _1,\overline{P}^2 _2)\ni(x,1)\mapsto x\in (P^2 _1,P^2 _2),\\
p_4:(T_{\overline{N}_2,1}(\overline{P}^2),T_{\overline{N}_2,2}(\overline{P}^2))\ni (x,t)\mapsto x\in (T_{N_2,1}(P^2),T_{N_2,2}(P^2))
\end{array}
$$
are projections. \cite[Lemma 6.1 (ii)]{BM2016} 
guaranties that inclusions $j_5$ and $j_6$ induce isomorphisms in cohomology. It is evident that projection $p_3$ induces an isomormphism, therefore, $p_4$ induces an isomorphisms too. Thus $I_{P^2}$ and $I_{\overline{P}^2}$ are conjugate and we have
\begin{equation}\label{cs2}
C(S_2,F)=C(\overline{S}_2,F^P).
\end{equation}
Since $\overline{T}_1\setminus\overline{T}_2=T_{\overline{N}_2,1}(\overline{P}^2)\setminus T_{\overline{N}_2,2}(\overline{P}^2)=(\overline{P}^2 _1\setminus\overline{P}^2 _2)\cap\Int \overline{N_2}$, inclusions $j_7$ and $i_{\overline{P^2},\overline{T}}$ are excisions; hence, they induce isomorphisms in cohomology. Consequently
\begin{equation}\label{eq:is2}
I_{\overline{P^2}}=H^*(F^P_{\overline{P^2},\overline{T}})\circ H^*(i_{\overline{P^2},\overline{T}})^{-1}.
\end{equation}

Now, put $K:=\underline{S}_1\cup\overline{S}_2$, $Q:=\underline{P}^1\cup\overline{P}^2$ and $M:=\underline{N}_1\cup\overline{N}_2$. 
By (\ref{fni}) %\cite[Corollary 7.3]{BM2016} 
we have
\begin{equation}\label{fni3}
F^P(\underline{N}_1)\cap \overline{N}_2=\emptyset \mbox{ and } F^P(\overline{N}_2)\cap \underline{N}_1=\emptyset.
\end{equation}
 This guaranties that $M$ is an isolating neighborhood of $K$ with respect to $F^P$, and $Q$ is a weak index pair in $M$.

Again we use \cite[Lemma 6.1(i)]{BM2016} to infer that $F$ is a map of pairs $(P _1,P _2)$ and $(T_{N,1}(P),T_{N,2}(P))$, and $F^P$ maps $(Q _1,Q _2)$ into $(T_{M,1}(Q),T_{M,2}(Q))$. In fact, as a consequence of (\ref{fni3}), $F^P(Q)\subset T_{M}(Q)\setminus ((\Int N_1\setminus P^1 _1)\times\{1\})\setminus ((\Int N_2\setminus P^2 _1)\times\{0\})=:Y$.
Consider the following commutative diagram
$$
\begin{tikzcd}
(P _1,P _2) \ar{r}{F} & (T_{N,1}(P),T_{N,2}(P)) & (P _1,P _2) \ar{l}[swap]{j_9} \\[3ex]
(Q_1,Q_2) \ar{u}{p_5} \ar{r}{F^P}\ar{rd}[swap]{F^P} & (Y_1,Y_2) \ar{u}{p_6}\ar{d}{j_{12}} & (Q_1,Q_2) \ar{l}[swap]{j_{10}}\ar{ld}{j_{11}}\ar{u}{p_5}\\[3ex]
&(T_{M,1}(Q),T_{M,2}(Q))&
\end{tikzcd}
$$
in which $j_9, j_{10}, j_{11}$ and $j_{12}$ are inclusions, 
$$
\begin{array}{l}
p_5:=p_1\cup p_3:(Q_1,Q_2)\ni(x,t)\mapsto x\in (P_1,P_2),\\
p_6:(Y_1,Y_2)\ni (x,t)\mapsto x\in (T_{N,1}(P),T_{N,2}(P)).
\end{array}
$$
are projections, and the contractions of $F^P$ to appropriate pairs are denoted simply by $F^P$. Inclusions $j_9$ and $j_{11}$ induce isomorphisms in cohomology, by \cite[Lemma 6.1(ii)]{BM2016}. Inclusions $j_{10}$ and $j_{12}$ also induce isomorphisms, as excisions, because $Y_1\setminus Y_2=T_{M,1}(Q)\setminus T_{M,2}(Q)=(Q_1\setminus Q_2)\cap\Int M$.  Since $P^1$ and $P^2$ are disjoint, and so are $\underline{P}^1$ and $\overline{P}^2$, projection $p_5$ is well defined and it induces an isomorphism in cohomology. Consequently, so does $p_6$. Therefore, $I_{P}$ and $I_{Q}$ are conjugate, and we have
\begin{equation}\label{cs3}
C(S,F)=C(K,F^P).
\end{equation}
Observe, that $\underline{M}:=\underline{W}\cup\overline{N_2}$ is an isolating neighborhood of $K$ with respect to $F^P$, and $\underline{Q}:=\underline{R}\cup\overline{P^2}$ is a weak index pair in $\underline{M}$. Now, using (\ref{eq:fp_R}) and (\ref{eq:fp_P2}), and arguing similarly as for the index maps $I_{\underline{R}}$ and $I_{\overline{P^2}}$, one can show that
\begin{equation}\label{eq:irp2}
I_{\underline{Q}}=H^*(F^P _{\underline{Q},\underline{T}\cup\overline{T}})\circ H^*(i_{\underline{Q},\underline{T}\cup\overline{T}})^{-1},
\end{equation}
where $i_{\underline{Q},\underline{T}\cup\overline{T}}:\underline{Q}\to \underline{T}\cup\overline{T}$ is the inclusion and 
$F^P _{\underline{Q},\underline{T}\cup\overline{T}}$ stands for the contraction of $F^P$ to the domain $\underline{Q}$ and codomain $\underline{T}\cup\overline{T}$.

Taking into account that $\underline{R}$ and $\overline{P^2}$ are disjoint pairs of spaces, we have $H^*(\underline{M})=H^*(\underline{R}\cup\overline{P}^2)=H^*(\underline{R})\times H^*(\overline{P}^2)$. 
Furthermore, $\underline{T}$ and $\overline{T}$ are disjoint, as desired. Therefore, $H^*(\underline{T}\cup\overline{T})=H^*(\underline{T})\times H^*(\overline{T})$ and, by (\ref{eq:irp2}), we have
$$
\begin{array}{rcl}
I_{\underline{Q}}&=&H^*(F^P _{\underline{Q},\underline{T}\cup\overline{T}})\circ H^*(i_{\underline{Q},\underline{T}\cup\overline{T}})^{-1}\\[1ex]
&=&\left(
H^*(F^P _{\underline{R},\underline{T}})\times
H^*(F^P _{\overline{P^2},\overline{T}})
\right)\circ\left(
H^*(i_{\underline{R},\underline{T}})^{-1}\times
H^*(i_{\overline{P^2},\overline{T}})^{-1}\right)
\\[1ex]
&=&\left(
H^*(F^P _{\underline{R},\underline{T}})\circ H^*(i_{\underline{R},\underline{T}})^{-1}
\right)
\times\left(
H^*(F^P _{\overline{P^2},\overline{T}})\circ H^*(i_{\overline{P^2},\overline{T}})^{-1}
\right),
\end{array}
$$
which, along with (\ref{eq:ir_ir}) and (\ref{eq:is2}), yields
$$
I_{\underline{Q}}=I_{\underline{R}}\times I_{\overline{P^2}}.
$$
Consequently, by \cite[Proposition 4.5]{M90} and (\ref{eq:lcs_1}), we have
$$
\begin{array}{rcl}
C(K,F^P)&=&L(H^*(\underline{Q}),I_{\underline{Q}})\\[1ex]
&=&L(H^*(\underline{R}),I_{\underline{R}})\times L(H^*(\overline{P^2}),I_{\overline{P^2}})\\[1ex]
&=&C(\underline{S}_1,F^P)\times C(\overline{S}_2,F^P).
\end{array}
$$     
This, along with (\ref{cs3}), (\ref{cs1}) and (\ref{cs2}), completes the proof.     
\hfill\qed
%-------------------------------------------

Notice that, if $F$ is single valued, then the disjointness of isolated invariant sets $S_1$ and $S_2$ implies condition (\ref{disa}). It is straightforward to see that this is also the case  for multivalued dynamical systems and strongly isolated invariant sets. 

However, the following example shows that in general the assertion (\ref{disa}) is essential for the additivity property of the Conley index.
\begin{ex}\label{ex:not_add}
{\em Let $X=[0,8]$ and let $F:X\mto X$ be given by
\begin{equation}\label{eq:f_add2}
F(x)=\left\{
\begin{array}{rl}
\{0\}&\for x\in [0,1)\\
{[0,1]}&\for x=1\\
\{1\}&\for x\in (1,2)\\
{[1,6]}&\for x=2\\
{[2,6]}&\for x\in (2,3)\\
{{2,7}}&\for x=3\\
\{7\}&\for x\in (3,4)\\
{[4,7]}&\for x=4\\
\{4\}&\for x\in (4,5)\\
{[4,6]}&\for x=5\\
{[5,6]}&\for x\in (5,6)\\
{[5,7]}&\for x=6\\
\{7\}&\for x\in (6,7)\\
{[7,8]}&\for x=7\\
\{8\}&\for x\in (7,8].
\end{array}
\right.
\end{equation}
The graph of $F$ is presented in Figure \ref{fig:add}. 
%--------------------------------
\begin{figure}[ht]
\begin{center}
\resizebox{80mm}{!}
{
\includegraphics{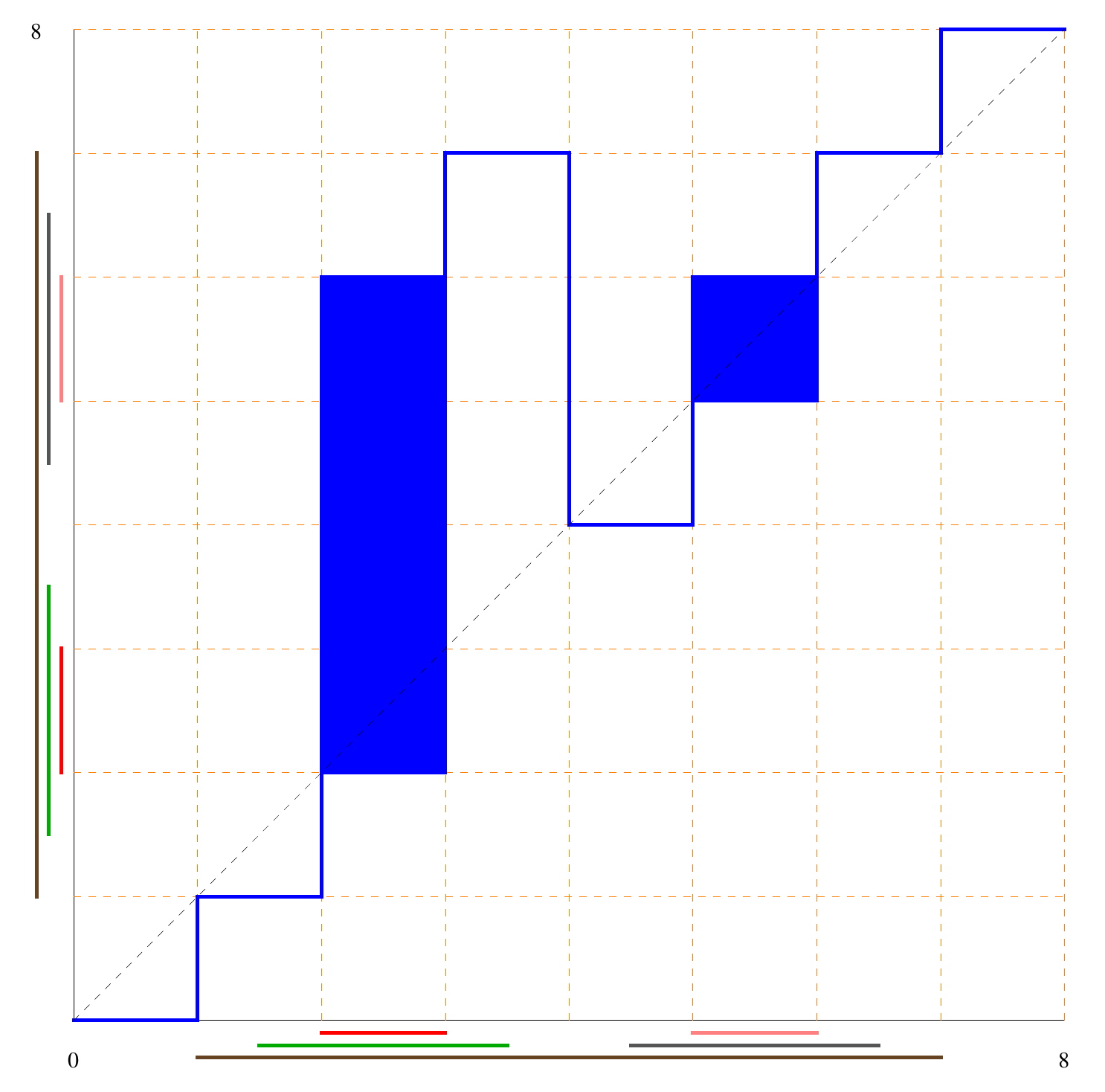}}
\caption{The graph of the map $F$ given by \eqref{eq:f_add2}, marked in blue.
Isolated invariant sets $S_1$ and $S_2$ are marked in red and pink, respectively. Isolating neighborhoods $N_1$ and $N_2$ for $S_1$ and $S_2$ are marked in green and gray, respectively.
The image $F(S_1)$, showing that $F(S_1)\cap S_2\neq\emptyset$, is marked in brown.
}
\label{fig:add}
\end{center}
\end{figure}
%--------------------------------
Consider isolated invariant sets $S_1:=[2,3]$ and $S_2:=[5,6]$, and isolating neighborhoods $N_1:=[1.5,3.5]$ and $N_2:=[4.5,6.5]$, respectively. Then $S:=S_1\cup S_2$ is an isolated invariant set with respect to $F$ and $N:=N_1\cup N_2$ isolates $S$.

Put $P^i _1:=N_i$ for $i=1,2$, $P^1 _2:=\{1.5,3.5\}$ and $P^2 _2:=\{4.5,6.5\}$. One can observe that $P^i:=(P^i _1,P^i _2)$ is a weak index pair in $N_i$, $i=1,2$.

It is easy to see that $H ^k (P^i)$ is trivial for $k\neq 1$ and that it has one generator for $k=1$. Moreover, $I_{P^i}=\id$. 
We have
$$
C_k(S_i,F)=\left\{\begin{array}{rl}
(\Z,\id )&\mbox{for } k=1\\[1ex]
0&\mbox{otherwise.}
\end{array}
\right.
$$ 
Observe that $P:=P^1\cup P^2$ is a weak index pair in $N$. Now, $H ^k (P)$ is trivial for $k\neq 1$ and it has two generators for $k=1$.
The index map is the isomorphism of the form (up to the order of generators)
$$
I_P=\left[\begin{array}{rr}
1&0\\
1&1
\end{array}
\right].
$$
We have
$$
C_k(S,F)=\left\{\begin{array}{rl}
(\Z ^2,I_P )&\mbox{for } k=1\\[1ex]
0&\mbox{otherwise.}
\end{array}
\right.
$$ 
Evidently,
$
C(S,F)\neq C(S_1,F)\times C(S_2,F).
$  
}
\end{ex}
%%%%%%%%%%%%%%%%%%%%%%%%%%%%%%%%%%%%%%%%%%%%%%%%%
\section{Homotopy property of the Conley Index}\label{sec:hom_prop}
As mentioned in the introduction we want the theory we develop to be useful for the reconstruction of the dynamics from a finite number of samples. For the details concerning the procedure of sampling the dynamics we refer to \cite{BM2016}. Let us just recall that the technique of sampling a given (usually unknown) dynamical system $f$ provides us with a generator $F$ of a discrete multivalued dynamical system. The procedure itself does not guarantee that $f$ is a selector of $F$. Even worse, $F$ does not have to contain any continuous selector, however we can expect that $f$ lies nearby $F$. We do not intend to artificially enlarge the values of $F$ in order to ensure that $F$ covers $f$, as it may result in loosing the isolation property.    
Our concept is to apply the Conley index theory directly to the multivalued dynamical system $F$ constructed by sampling the dynamics $f$, and then try to extend the results to the underlying unknown $f$. For this purpose the homotopy (continuation) property of the Conley index is crucial.    

Let $X$ be a locally compact metrizable space, let $\Lambda\subset \R$ be a compact interval, and let an upper semicontinuous mapping $F:\Lambda\times X\mto X$ with compact values be determined by a given morphism. 
Assume that, for each $\lambda\in\Lambda$, $F_\lambda:X\mto X$, given by $F_\lambda(x):=F(\lambda,x)$, is a discrete multivalued dynamical system. The family $\{F_{\lambda}\}$ will be referred to as a parameterized family of discrete multivalued dynamical systems.

We will simply write $\lambda$ instead of $F_\lambda$ whenever $F_\lambda$ appears as a parameter. According to this assumption, given a compact subset $N\subset X$ and $\lambda\in\Lambda$, the sets $\Inv ^{(\pm)}N$ with respect to $F_\lambda$ are denoted by $\Inv^{(\pm)}(N,\lambda)$.

The main result of this section is the following theorem. 
%-----------------------------------------
\begin{thm}{\em (}{\bf Homotopy} {\em(}{\bf continuation}{\em)} {\bf property}{\em)}\label{thm:homotopy}
Let $\Lambda\subset\R$ be a compact interval and let $F_\lambda:X\mto X$ be a parameterized family of dmds. If $N$ is an isolating neighborhood for each $\lambda\in\Lambda$ then $C(\Inv(N,\lambda ))$ does not depend on $\lambda\in\Lambda$.  
\end{thm}
We postpone the proof of the theorem to the end of this section.

Let us recall that, for any compact $N\subset X$, the mappings $\lambda\mapsto \Inv^{+}(N,\lambda)$, $\lambda\mapsto \Inv^{-}(N,\lambda)$ and $\lambda\mapsto \Inv(N,\lambda)$ are upper semicontinuous on $\Lambda$ (cf. \cite[Lemma 4.2]{KM95}).

Using exactly the same arguments as in the single valued case (cf. \cite[Corollary 7.3]{M90}) one can prove the following lemma.  

\begin{lm}\label{lm:hom_l}
Suppose $N\subset X$ is an isolating neighborhood for $F_{\lambda _0}$. Then $N$ is an isolating neighborhood for $F_{\lambda}$, for $\lambda\in\Lambda$ sufficiently close to $\lambda_0$.
\end{lm}
\cite[Theorem 4.12]{BM2016} states that 
any isolating neighborhood $N$ admits a weak index pair. Since we will use it frequently, for convenience we quote it here in a slightly modified, adapted to our needs form. We use the following notation. Given $N\subset X$ and an interval $I$ in $\ZZ$ set
\[
   \Sol(N,F_\lambda,I):=\setof{\sigma:I\to N\text{ a solution  for $F_\lambda$}}
\]
and for $x\in N$ and $n\in\Z^+$ define
\begin{eqnarray*}
F_{\lambda,N,n}(x)&:=& \{y\in N\mid \exists\,\sigma\in\Sol(N,F_\lambda,[0,n]):\,\sigma(0)=x,\sigma(n)=y\},\\
F_{\lambda,N,-n}(x)&:=&\{y\in N\mid \exists\,\sigma\in\Sol(N,F_\lambda,[-n,0]):\,\sigma(-n)=y,\sigma(0)=x\}, \\
F_{\lambda,N} ^+(x)&:=&\bigcup_{n\in\z^+}F_{\lambda,N,n}(x),\\
F_{\lambda,N} ^-(x)&:=&\bigcup_{n\in\z^+}F_{\lambda,N,-n}(x).
\end{eqnarray*}
%--------------------------------------------
\begin{thm}\label{thm:ex_wipBM}
Let $N$ be an isolating neighborhood for $F$ and let $U$ and $V$ be open neighborhoods of $\Inv^{+}N$ and $\Inv^{-}N$, respectively, with $U\cap V\subset \Int N$. Then, there exists a compact neighborhood $A$ of $\Inv^{-}N$ such that $P_1:=F_N ^+(A)\subset V$. Moreover, $P:=(P_1,P_2)$ where $P_2:=F_N ^+(P_1\setminus U)$ is a weak index pair in $N$ and $P_1\setminus P_2\subset U$. {\rm (see \cite[Theorem 4.12]{BM2016} and its proof)}.
\end{thm}
%---------------------------------------------
We need the following lemma.
\begin{lm}\label{lem:nwips}
If $N$ is an isolating neighborhood with respect to $F$ then, for any $n\in\N$, there exist weak index pairs $P^1,\dots, P^n$ in $N$ such that $P^i\subset\Int_N P^{i+1}$ for $i=1,\dots,n-1$.  
\end{lm}
\begin{proof} 
Take open neighborhoods $U$ and $V$ of $\Inv^{+}N$ and $\Inv^{-}N$, respectively, with $U\cap V\subset \Int N$. By Theorem \ref{thm:ex_wipBM} there exists a compact neighborhood $A$ of $\Inv^{-}N$ such that the sets $P_1:=F_N ^+(A)$ and $P_2:=F_N ^+(P_1 ^n\setminus U)$ form a weak index pair $(P_1,P_2)$ in $N$. Moreover, $P_1\subset V$ and $P_1\setminus P_2\subset U$.

We shall define, recurrently, a sequence of compact sets $Q_2 ^i$ and a corresponding descending sequence $U_i$ of open neighborhoods of $\Inv^{+}N$, for $i=1,2\dots ,n$, such that
\begin{equation}\label{ink_1}
\left\{
\begin{array}{rl}
\mbox{(i)}&Q_2 ^{i}=F_N ^+(P_1\setminus U_{i})\for i=1,2,\dots ,n,\\[1ex]
\mbox{(ii)}&Q_2 ^{i-1}\subset\Int_{P_1}Q_2 ^{i}\for i=2,\dots ,n,\\[1ex]
\mbox{(iii)}&(P_1,Q_2 ^{i})\;\mbox{ is a weak index pair in }N \mbox{ for }i=1,2,\dots ,n. 
\end{array}
\right.
\end{equation} 
Put $U_1:=U$ and $Q_2 ^1:=P_2$ and observe that (\ref{ink_1}) is satisfied for $i=1$. Next, fix $k\in\{1,2,\dots,n-1\}$ and suppose that the sequences $Q_2 ^i$ and $U_i$ satisfying (\ref{ink_1}) are defined for $i=1,2,\dots, k$. Observe that there exists a neighborhood $W_k$ of $Q_2 ^k$ open in $N$ such that
$$
Q_2 ^k\subset W_k\subset\Cl W_k\subset N\setminus\Inv ^+N
$$
and 
$$
N\setminus U_k\subset\cl W_k.
$$
We set $U_{k+1}:=N\setminus\cl W_k$. Then $U_{k+1}$ is an open neighborhood of $\Inv^{+}N$, and $U_{k+1}\subset U_{k}$. We have $U_{k+1}\cap V\subset U_{k}\cap V\subset U\cap V\subset\Int N$. We define $Q_2 ^{k+1}:=F_N ^+(P_1\setminus U_{k+1})$. Then, by Theorem \ref{thm:ex_wipBM}, $(P_1,Q_2 ^{k+1})$ is a weak index pair in $N$. It remains to verify that inclusion (\ref{ink_1})(ii) holds for $i=k+1$. Indeed, we have 
$$
Q_2 ^k\subset W_k\cap P_1\subset\cl W_k\cap P_1\subset Q_2 ^{k+1};
$$
hence, inclusion (\ref{ink_1}) for $i=k+1$ follows.

We shall prove that there exists a sequence $(P_1 ^i, R_2 ^i)$, for $i=1,2,\dots, n$, of weak index pairs such that
\begin{equation}\label{ink_2}
P_1 ^i\subset \Int _N P_1 ^{i+1}\for i=1,2\dots,n-1,
\end{equation}
and 
\begin{equation}\label{ink_31}
R_2 ^i=F_N ^+(P_1 ^{i}\setminus U).
\end{equation}
Indeed, we define $(P_1 ^n,R_2^n):=(P_1,Q_2 ^1)$. Choose $V_n$, an open set such that $V_n\cap N=\Int _N P_1 ^{n}$. Without loss of generality we may assume that $(V_n\cap U)\setminus N=\emptyset$. Observe that then $V_n$ is an open neighborhood of $\Inv^- N$, and $V_n\cap U\subset V\cap U\subset \Int N$. Therefore, applying Theorem \ref{thm:ex_wipBM} we construct a weak index pair $(P_1 ^{n-1}, R_2 ^{n-1})$ such that $P_1 ^{n-1}\subset V_n\cap N=\Int _N P_1 ^{n}$ and $R_2 ^{n-1}=F_N ^+(P_1 ^{n-1}\setminus U)$. By the reverse recurrence we are
done.

Finally we put 
$$
P_2 ^i:=Q_2 ^i\cap P_1 ^i\for i=1,2,\dots,n.
$$
By (\ref{ink_31}) and (\ref{ink_1}), for each $i=1,2,\dots,n$, we have 
$R_2 ^i=F_N ^+(P_1 ^{i}\setminus U)\subset F_N ^+(P_1 \setminus U)=Q_2 ^1\subset Q_2 ^i$; hence,  according to (\ref{ink_2}), we have the inclusion $(P_1 ^i,R_2 ^i)\subset (P_1 ^n, Q_2 ^i)$ of weak index pairs. Therefore, by \cite[Lemma 5.4]{BM2016}, we infer that $(P_1 ^i,P_2 ^i)$ is a weak index pair. 

By (\ref{ink_1}) and (\ref{ink_2}), for each $i=1,2,\dots,n$, we have the following inclusions
$$
P_2 ^i=Q_2 ^i\cap P_1 ^i\subset\Int _{P_1 ^n}Q_2 ^{i+1}\cap \Int_NP_1 ^{i+1}\subset Q_2 ^{i+1}\cap P_1 ^{i+1}=P_2 ^{i+1}, 
$$
showing that $P_2 ^i\subset \Int_N P_2^{i+1}$. This, along with (\ref{ink_2}), completes the proof.
\qed\end{proof}
%---------------------------------------------
\begin{lm}\label{lem:wip_l}
Assume that $N$ is an isolating neighborhood with respect to $F_\mu$ for some $\mu\in\Lambda$ and $P,Q,R$ are weak index pairs with respect to $F_\mu$ such that $P\subset\Int_N Q$, $Q\subset\Int_N R$. Then there exists $\Lambda_0$, a neighborhood of $\mu$ in $\Lambda$, such that for every $\lambda\in\Lambda_0$ there exists a weak index pair $P(\lambda)$ with respect to $F_\lambda$ satisfying $P\subset P(\lambda)\subset R$.
\end{lm}
\begin{proof}
Put $P_1 '(\lambda):=F_{\lambda,N} ^+(P_1)$, $P_2(\lambda):=F_{\lambda,N} ^+(Q_2)$ and $P_1 (\lambda):=P_1 '(\lambda)\cup P_2(\lambda)$. We will prove that for $\lambda$ sufficiently close to $\mu$ the pair $P(\lambda):=(P_1(\lambda),P_2(\lambda))$ satisfies the assertions of the lemma.

Using similar reasoning as in the proof of \cite[Lemma 7.4]{M90}, which we present here for the sake of completeness, one can show that there exists $\Lambda_1$, a neighborhood of $\mu$ in $\Lambda$ such that for every $\lambda\in\Lambda_1$ we have 
\begin{equation}\label{inv-lambda}
\Inv^-(N,\lambda)\subset P_1,
\end{equation}
\begin{equation}\label{inv+lambda}
\Inv^+(N,\lambda)\cap Q_2=\emptyset,
\end{equation}
\begin{equation}\label{ink4}
P_1 '(\lambda)\subset Q_1\mbox{ and }P_2(\lambda)\subset R_2.
\end{equation}
Put $Z:=N\setminus \Int_N Q_1$. By \cite[Lemma 4.2]{KM95} and Lemma \ref{lm:hom_l} one can find a compact neighborhood $\Delta$ of $\mu$ in $\Lambda$ such that, for any $\lambda\in\Delta$, $N$ is an isolating neighborhood with respect to $F_\lambda$ and properties (\ref{inv-lambda}), (\ref{inv+lambda}) hold. Since $P\subset \Int _N Q$, we have
\begin{equation}
\Inv^-(N,\lambda)\cap Z=\emptyset\mbox{ and }
\Inv^+(N,\lambda)\cap P_2=\emptyset\for \lambda\in \Delta.
\end{equation}
Define $G:\Delta\times X\ni (\lambda,x)\mapsto (\lambda, F(\lambda,x))\subset \Delta\times X$ and $M:=N\times \Delta$. One can verify that $M$ is an isolating neighborhood with respect to $G$ and
$$
\begin{array}{l}
\Inv ^-(M,G)\cap (Z\times \Delta)=\bigcup\{\Inv^-(N,\lambda)\cap (Z\times \{\lambda\})\mid\lambda\in\Lambda\}=\emptyset,\\[1ex]
\Inv ^+(M,G)\cap (P_2\times \Delta)=\bigcup\{\Inv^+(N,\lambda)\cap (P_2\times \{\lambda\})\mid\lambda\in\Lambda\}=\emptyset.
\end{array}
$$
Let $x\in Z$. Then $F_{\mu,N} ^-(x)\cap P_1=\emptyset$ and $(x,\mu)\in Z\times\Delta$, i.e. $G_M ^-(x,\mu)\cap (P_1\times\Delta)=\emptyset$. By \cite[Lemma 2.10 (b)]{KM95}, $G_M ^-$ is upper semicontinuous. Hence, there exist an open neighborhood $V_x$ of $x$ in $N$ and $\Delta_x$ of $\mu$ in $\Delta$ such that
$$
G_M ^-(y,\lambda)\cap (P_1\times\Delta)=\emptyset\for y\in V_x, \lambda\in \Delta_x,
$$  
i.e.
$$
F_{\lambda,N} ^-(y)\cap P_1=\emptyset\for y\in V_x, \lambda\in \Delta_x.
$$
Since $Z$ is compact, there exists a finite subset $Z_0\subset Z$ such that $Z\subset \bigcup\{V_x\mid x\in Z_0\}$. Put $\Delta_0:=\bigcap\{\Delta_x\mid x\in Z_0\}$. Then $\Delta_0$ is a neighborhood of $\mu$ and $F_{\lambda,N} ^-(y)\cap P_1=\emptyset$ for $y\in Z$ and $\lambda\in \Delta_0$. This shows that $P_1 '(\lambda)=F_{\lambda,N} ^+(P_1)\subset N\setminus Z=\Int _N Q_1\subset Q_1$ for $\lambda\in \Delta_0$.

Now take $x\in Q_2$. Then $F_{\mu,N} ^+(x)\subset Q_2\subset\Int R_2$. Similar reasoning using compactness of $Q_2$ and the mapping $G$ shows that there exists a neighborhood $\Delta_1$ of 
$\mu$ such that $F_{\lambda,N} ^+(y)\subset\Int R_2$ for $y\in Q_2$ and $\lambda\in \Delta_1$. Thus $P_2(\lambda)=F_{\lambda,N}^+(Q_2)\subset R_2$ for $\lambda\in \Delta_1$.

In conclusion, conditions (\ref{inv-lambda}), (\ref{inv+lambda}) and (\ref{ink4}) hold for every $\lambda\in \Lambda_1:=\Delta\cap\Delta_0\cap\Delta_1$.  

By \cite[Lemma 5.4]{BM2016}, $(P_1\cup R_2,R_2)$ is a weak index pair in $N$ with respect to $F_\mu$, hence $\Inv(N,\mu)\subset\Int(P_1\setminus R_2)$. By the upper semicontinuity of $\lambda\mapsto \Inv(N,\lambda)$, there exists $\Lambda_2$, a neighborhood of $\mu$ in $\Lambda$ such that for every $\lambda\in \Lambda _2$ we have
\begin{equation}\label{ink3}
\Inv(N,\lambda)\subset\Int(P_1\setminus R_2).
\end{equation} 

We shall prove that $P(\lambda)$ is a weak index pair in $N$ with respect to $F_\lambda$ for every $\lambda\in\Lambda_0:=\Lambda_1\cap\Lambda_2$ . 

The compactness of $P_2(\lambda)$ follows from (\ref{inv+lambda}) and \cite[Lemma 2.10]{KM95}. By (\ref{inv-lambda}) and \cite[Lemma 2.9]{KM95}, $P_1 '(\lambda)$ is compact, hence so is $P_1 (\lambda)$. The inclusion $P_2 (\lambda)\subset P_1 (\lambda)$ is obvious. 

The positive invariance of $P(\lambda)$ in $N$ (property (a)) with respect to $F_\lambda$ is straightforward. 

By (\ref{ink3}) we have $\Inv(N,\lambda)\subset\Int(P_1\setminus R_2)\subset\Int(P_1(\lambda)\setminus P_2(\lambda))$ which means that condition (c) holds.

We have $P_1(\lambda)\setminus P_2(\lambda)=P_1 '(\lambda)\setminus P_2(\lambda)\subset  Q_1\setminus P_2(\lambda)\subset Q_1\setminus Q_2\subset\Int N$, thus condition (d) is verified.    

We need to verify condition (b). Suppose the contrary and consider $x\in \bd _{F_\lambda}P_1(\lambda)\setminus P_2(\lambda)$. By property (d) of $P(\lambda)$ we have $x\in\Int N$. Since $x\in\cl (F_\lambda(P_1(\lambda ))\setminus P_1(\lambda))$, there exists $y\in (F_\lambda(P_1(\lambda ))\setminus P_1(\lambda))\cap\Int N$. Take $u\in P_1(\lambda)$ such that $y\in F_\lambda(u)$. If $u\in P_1 '(\lambda)=F_{\lambda,N} ^+(P_1)$ then there exists a solution $\sigma_\lambda:[0,n]\to N$ with respect to $F_\lambda$ such that $\sigma_\lambda (0)\in P_1$ and $\sigma_\lambda (n)=u$. Since $y\in F_\lambda (u)\cap N$, we can put $\sigma_\lambda (n+1):=y$ to infer that $y\in P_1 '(\lambda)$. This yields $y\in P_1(\lambda)$, a contradiction. In the case where $u\in P_2 (\lambda)=F_{\lambda,N} ^+(Q_2)$,  there exists a solution $\sigma_\lambda:[0,n]\to N$ such that $\sigma_\lambda (0)\in Q_2$ and $\sigma_\lambda (n)=u$. We can extend $\sigma_\lambda$ putting $\sigma_\lambda (n+1):=y$. Since $y\in F_\lambda (u)\cap N$, we obtain $y\in P_2 (\lambda)$ and, consequently, $y\in P_1(\lambda)$, a contradiction.  
 
Finally, inclusions (\ref{ink4}) guarantee that $P\subset P(\lambda)\subset R$ for $\lambda\in\Lambda_0$.
\qed\end{proof}
%---------------------------------------------
\begin{lm}\label{lem:morph}
Assume that $N$ is an isolating neighborhood with respect to $F_\mu$ for some $\mu\in\Lambda$ and $P^1,\dots,P^7$ are weak index pairs with respect to $F_\mu$ such that $P^i\subset\Int P^{i+1}$, $i=1,\dots,6$. Then there exists $\Lambda_0$, a neighborhood of $\mu$ in $\Lambda$, such that for every $\kappa\in\Lambda_0$ there exists a weak index pair $P(\kappa)$ with respect to $F_\kappa$ satisfying $P^1\subset P(\kappa)\subset P^7$ and such that the inclusions
$$
i:P^1\to P(\lambda),\mbox{\hspace{1cm}} j:P(\lambda)\to P^7 
$$
induce morphisms in the category of endomorphisms
$$
i^*:H^*(P(\kappa),I_{P(\kappa)})\to H^*(P^1,I_{P^1}), 
$$
$$
j^*:H^*(P^7,I_{P^7})\to H^*(P(\kappa),I_{P(\kappa)}). 
$$
\end{lm}
\begin{proof}
By Lemma \ref{lem:wip_l} we can find a neighborhood $\Lambda_0$ of $\mu$ in $\Lambda$ such that for each $\lambda\in\Lambda_0$ there exist weak index pairs $Q(\lambda)$, $P(\lambda)$ and $R(\lambda)$ with respect to $F_\lambda$ such that
$$
P^1\subset Q(\lambda)\subset P^3\subset P(\lambda)\subset P^5\subset R(\lambda)\subset P^7.
$$
Now, using the same arguments as in the proof of \cite[Lemma 7.5]{M90}, one can show that for arbitrarily fixed $\lambda\in \Lambda_0$, the weak index pair $P(\lambda)$ satisfies the assertion of the lemma. 

Indeed, fix $\kappa\in \Lambda_0$ and consider the following, generally noncommutative diagram
$$
\begin{tikzcd}
P^1\ar{d}\ar{r}{F_\mu} & T_N(P^1) \ar{d}\\[3ex]
P(\kappa)\ar{r}{F_\kappa}\ar{d} & T_{N}(P(\kappa))\ar{d}\\[3ex]
P^7\ar{r}{F_\mu} & T_N(P^7)
\end{tikzcd}
$$
in which horizontal arrows denote inclusions. By Lemma \ref{l1}, for any $\lambda\in \Lambda_0$, we have
$$
\begin{array}{l}
F_\lambda(P^1)\subset F_\lambda(Q(\lambda))\subset T_N(Q(\lambda))\subset T_N(P^3)\subset T_N(P(\lambda)),\\[1ex]
F_\lambda(P(\lambda))\subset F_\lambda(P^5)\subset F_\lambda (R(\lambda))\subset T_N(R(\lambda))\subset T_N(P^7)
\end{array}
$$
and $\mu$ and $\lambda$ are joined in $\Lambda _0$ by an arc. Therefore, we get the commutative diagram in the Alexander--Spanier cohomology
$$
\begin{tikzcd}
H^*(P^1) & [4em]\ar{l}[swap]{H^*(F_{\mu,P^1})} H^*(T_{N}(P^1)) \ar{r}{H^*(i_{P^1})}&[4em] H^*(P^1 ) \\[3ex]
H^*(P(\kappa)) \ar{u}{i^*}& \ar{l}[swap]{H^*(F_{\kappa,P(\kappa)})} H^*(T_{N}(P(\kappa)))\ar{u}{}\ar{r}{H^*(i_{P(\kappa)})}& H^*(P(\kappa))\ar{u}{i^*}  \\[3ex]
H^*(P^7) \ar{u}{j^*}&\ar{l}[swap]{H^*(F_{\mu,P^7})}  H^*(T_{N}(P^7))\ar{u}{}\ar{r}{H^*(i_{P^7})}& H^*(P^7 ) \ar{u}{j^*}
\end{tikzcd}
$$
which, according to the definition of an index map, completes the proof. 
\qed\end{proof}\ \ \

%----------------------------------------------
{\bf Proof of Theorem \ref{thm:homotopy}:} 
For the proof we use Lemma \ref{lem:nwips}, Lemma \ref{lem:wip_l} and Lemma \ref{lem:morph}, and argue similarly as in the proof of \cite[Theorem 2.11]{M90}. We present this reasoning here for the sake of completeness.

It suffices to show that for any given $\mu\in \Lambda$ there exists its neighbourhood $\Lambda_0$, such that for all $\nu\in\Lambda _0$
$$
C(\Inv (N,\mu))= C(\Inv (N,\nu)).
$$ 
Thus, fix $\mu\in\Lambda$. By Lemma \ref{lem:nwips} we can take weak index pairs
$P^1,\dots,P^{13}$ with respect to $F_\mu$ such that $P^i\subset\Int P^{i+1}$, $i=1,\dots,12$. Using Lemma \ref{lem:morph} for $P^1,\dots,P^7$ and again for $P^7,\dots, P^{13}$, we infer that there exists a neighborhood $\Lambda_0$ of $\mu$ such that for every $\lambda\in\Lambda_0$ there exist weak index pairs $P(\lambda)$, $Q(\lambda)$ with respect to $F_{\lambda}$, with $P^1\subset P(\lambda)\subset P^7\subset Q(\lambda)$ and such that we have the following commutative diagram of maps induced by inclusions
$$
\begin{tikzcd}
(H^*(P(\lambda)),I_{P(\lambda)})\ar{d}[swap]{j_0} &\ar{l}[swap]{j_1}(H^*(P^7),I_{P^7})\\[3ex]
(H^*(P^1),I_{P^1})& \ar{l}{j} (H^*(Q(\lambda)),I_{Q(\lambda)})\ar{u}[swap]{j_2}.\\
\end{tikzcd}
$$
Applying the Leray functor to the above diagram we get from \cite[Theorem 6.4]{BM2016} that $L(j_0)\circ L(j_1)=L(j_0\circ j_1)$ and $L(j_1)\circ L(j_2)=L(j_1\circ j_2)$ are isomorphisms; hence, so is $L(j)$. Therefore,
$$
C(\Inv (N,\mu))=L(H^*(P^1),I_{P^1})=L(H^*(Q(\lambda)),I_{Q(\lambda)})=C(\Inv (N,\lambda)),
$$  
which completes the proof.\qed
%----------------------------------------------

At the end of this section we revisit the leading example of \cite{EdJaMr15} and \cite{BM2016} in order to illustrate the homotopy property of the Conley index.

\begin{ex}\label{ex:hom1}
{\em
Take $S^1=\RR/\ZZ$. For simplicity we will identify a real number $x\in\RR$ with its equivalence class $[x]\in \RR/\ZZ$.
Consider the self-map
\begin{equation}
\label{eq:f}
   f: S^1\ni x\mapsto 2x\in S^1.
\end{equation}
Let $x_i:=\frac{i}{16}$. We take $A:=\setof{x_i\mid  i=0,1,2,\dots,15}\subset S^1$ as the finite sample.
Since $f(A)\subset A$, the restriction $g:=f_{|A}$ is an exact sample of $f$ on $A$.
Consider the grid $\cA$ on $S^1$ consisting of intervals $[x_i-\frac{1}{32},x_i+\frac{1}{32}]$.
The graph of the multivalued map $F$ obtained from the smallest combinatorial enclosure
$\cG$ of $g$ on $\cA$ is presented in Figure~\ref{fig:z2on16samples}.
Note that $F$ does not admit a continuous selector.

%--------------------------------
\begin{figure}[ht]
\begin{center}
\includegraphics[width=0.7\textwidth]{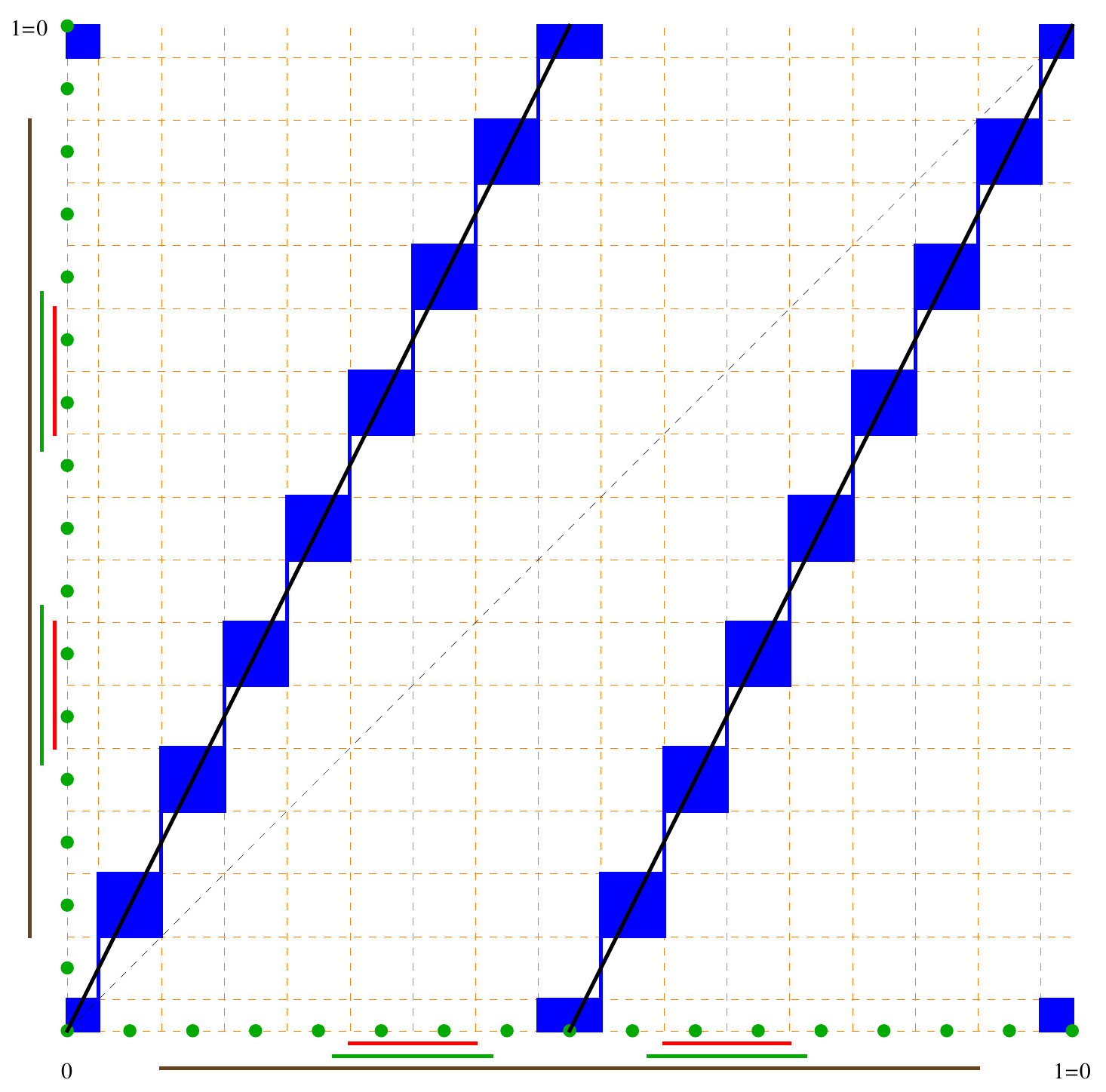}
\caption{The graph of the map $f$ given by \eqref{eq:f}, marked in black, and its sampling.
The $16$ sampled points are marked with green dots. The grid consisting of 16 intervals is marked in orange.
The graph of $F$ constructed from the sampling points is marked in blue.
A candidate for an isolated invariant set $S$ is marked in red. Its image $F(S)$, showing that $S$ is not a strongly
isolated invariant set, is marked in brown.
An isolating neighborhood $N$ for $S$ is marked in green.
}
\label{fig:z2on16samples}
\end{center}
\end{figure}
%--------------------------------

  Observe that $0$ is a hyperbolic fixed point of $f$. Thus, $\{0\}$ is an isolated invariant set of $f$. It belongs to $S:=[\frac{31}{32},\frac{1}{32}]\in\cA$. It is straightforward to observe that $S$ is an invariant set for $F$. 
  
 Consider 
\begin{equation}\label{eq:cont_ex}
F_\lambda(x):=\lambda F(x)+(1-\lambda)f(x)\for x\in S^1,\lambda\in [0,1]:=\Lambda
\end{equation}
and $N:=[\frac{15}{16},\frac{1}{16}]$
in order to see that $\{0\}=\Inv (N,F_0)$ and $S=\Inv (N, F_1)$ are related by continuation (cf. e.g. \cite{C78,RS88}). Therefore, by the homotopy property of the Conley index, it follows that the Conley index of $S$ for $F$ is the same as the Conley index of $\{0\}$ for $f$. 
  
Recall, that in \cite[Example 8.1]{BM2016} the same conclusion was obtained by the direct computation of the indices.
}
\end{ex}
%-----------------------------------------------
\begin{ex}\label{ex:hom2}
{\em
Le us consider dynamical systems $f$, $F$ and $F_\lambda$ defined in the preceding example. Note that $\{\frac{1}{3},\frac{2}{3}\}$ is a hyperbolic periodic trajectory of $f$. In particular, it is
an isolated invariant set for $f$. Consider the cover of this set by elements of the grid $\cA$ and set
$S:=[\frac{9}{32},\frac{13}{32}]\cup [\frac{19}{32},\frac{23}{32}]$. It is easy to see that $S$ is an invariant set for $F$, each point of which belongs to a $2$-periodic trajectory of $F$ in $S$. 

Considering  the dmds $F_\lambda$ given by (\ref{eq:cont_ex}) and
$N:=[\frac{17}{64},\frac{27}{64}]\cup [\frac{37}{64},\frac{47}{64}]$ one can observe that 
$\{\frac{1}{3},\frac{2}{3}\}$ and $S$ are related by continuation. Thus, the homotopy property of the Conley index guaranties that the Conley index of $S$ for $F$ coincides with the Conley index
of $\{\frac{1}{3},\frac{2}{3}\}$ for $f$. 

For  the direct verification of the mentioned coincidence resulting with
$$
C_k(S,F)=C_k(\{\frac{1}{3},\frac{2}{3}\},f)=\left\{\begin{array}{rl}
(\Z ^2,\tau)&\mbox{for } k=1\\[1ex]
0&\mbox{otherwise,}
\end{array}
\right.
$$
where $\tau$ is a transposition $\tau:\Z ^2\ni(x,y)\mapsto (y,x)\in \Z^2,$
we refer to \cite[Example 8.2]{BM2016}).
}
\end{ex}
%%%%%%%%%%%%%%%%%%%%%%%%%%%%%%%%%%
\section{Commutativity property of the Conley index}\label{sec:com_prop}
In this section we discuss another intrinsic property of the indices of Conley type, namely the commutativity property. It seems that this issue, in the context of discrete multivalued dynamical systems, appears here for the first time.

Our goal in this section is to prove a multivalued counterpart of \cite[Theorem 1.12]{M94}.    

Throughout this section $X$ and $Y$ are assumed to be locally compact metrizable spaces.
 
\begin{thm}{\em (}{\bf Commutativity property}{\em)}\label{thm:com1}
Let $F:X\mto X$ and $G:Y\mto Y$ be given discrete multivalued dynamical systems and let $\varphi:X\to Y$ and $\Psi:Y\mto X$ be partial maps such that $\varphi$ is continuous and injective, $\dom\varphi$ is compact, $\Psi$ is upper semicontinuous, $F=\Psi\varphi$ and $G=\varphi\Psi$. Assume that $S\subset X$ is an isolated invariant set with respect to $F$, and the domains of $\varphi$ and $\Psi$ contain neighborhoods of $S$ and $\varphi (S)$, respectively. Then $\varphi (S)$ is an isolated invariant set with respect to $G$ and $C(S,F)=C(\varphi (S),G)$. 
\end{thm}
Before we proceed with the proof, let us remark that the assertion that $\varphi$ is injective, is essential. Namely, unlike in the single valued case, the image of an isolated invariant set $S$ under a noninjective $\varphi$ need not be an isolated invariant set, as presented in Example \ref{ex:com1}. Moreover, even if the image $\varphi (S)$ is an isolated invariant set, its Conley index may differ from the Conley index of $S$ (see Example \ref{ex:com2}). However, if the underlying isolated invariant set $S$ is  strongly isolated, then the commutativity property holds true for an arbitrary continuous $\varphi$, not necessarily injective (see Theorem \ref{thm:com2}).\\

{\bf Proof of Theorem \ref{thm:com1}:} First we observe that $\varphi (S)$ is invariant with respect to $G$. Indeed, consider $y\in \varphi (S)$ and take $x\in S$ such that $\varphi (x)=y$. Since $S$ is invariant with respect to $F$, there exists a solution $\sigma:\Z\to S$ with $\sigma (0)=x$. Define $\tau (k):=\varphi(\sigma (k))$ for $k\in\Z$. We have $\tau (0)=y$ and $\tau (k)=\varphi(\sigma (k))\in \varphi (F(\sigma (k-1))=\varphi (\Psi(\varphi (\sigma (k-1)))=\varphi (\Psi(\tau (k-1))=G(\tau (k-1))$, for an arbitrary $k\in\Z$. This means that $\tau$ is a solution with respect to $G$ through $y$ in $\varphi(S)$. We have proved that $\varphi(S)\subset \Inv _G(\varphi(S))$. The opposite inclusion is obvious.
  
Let $M'\subset\dom\varphi$ be an isolating neighborhood of $S$. 

Since $\varphi (S)$ and $\varphi (\dom {\varphi}\setminus\Int M')$ are closed and disjoint, 
%By (\ref{eq:com_sep}) 
we can take a compact neighborhood $N$ of $\varphi (S)$ such that 
\begin{equation}\label{eq:def_N}
N\cap\im\varphi\subset N\cap\varphi (M'). 
\end{equation}
Assume without loss of generality that $N\subset\dom\Psi$.

We shall prove that $N$ is an isolating neighborhood of $\varphi (S)$ with respect to $G$. Since $\varphi (S)$ is invariant with respect to $G$ and $ \varphi (S)\subset \Int N$, it suffices to verify that $\Inv_GN\subset\varphi (S)$. To this end consider $y\in \Inv _GN$ and $\tau:\Z\to N$, a solution with respect to $G$ through $y$, i.e. $\tau (0)=y$ and $\tau (k)\in G(\tau(k-1))$, for $k\in \Z$. Without loss of generality one can assume that $\tau (\Z)\subset N\cap\im\varphi$, as $G=\varphi\Psi$. Put $\sigma:=\varphi ^{-1}\tau$ and observe that $\sigma (k)=\varphi ^{-1}(\tau (k))\in \varphi ^{-1}(G(\tau (k-1)))=F(\sigma(k-1))$, for $k\in\Z$. Moreover, by (\ref{eq:def_N}), $\sigma(\Z)\subset M'$, which means that $\sigma$ is a solution with respect to $F$ in $M'$. But $M'$ is an isolating neighborhood of $S$ with respect to $F$, hence $\sigma(\Z)\subset S$ and, as a consequence, $\tau (\Z)\subset \varphi (S)$. In particular $y\in \varphi (S)$. 
 
It is straightforward to observe that $M:=\varphi ^{-1}(N)\subset M'$ is an isolating neighborhood of $S$ with respect to $F$. Take a weak index pair $Q=(Q_1,Q_2)$ in $N$ and define $P_i:=\varphi ^{-1}(Q_i)$, $i=1,2$. We will prove that $P:=(P_1,P_2)$ is a weak index pair in $M$. Compactness of $P_1$ and $P_2$ as well as inclusions $P_2\subset P_1\subset M$ are obvious. For the proof of property (a) take $x\in F(P_i)\cap M$. Then $\varphi(x)\in\varphi(F(P_i)\cap M)=G(Q_i)\cap N$. By the positive invariance of $Q_i$ in $N$ with respect to $G$ we have $\varphi (x)\in Q_i$, hence $x\in P_i$. For the proof of (b) take $x\in\bd _F(P_1)$. Then $\varphi(x)\in Q_1\cap\cl (\varphi (F(P_1))\setminus Q_1)$, as $\varphi$ is injective. Consequently, $\varphi(x)\in Q_1\cap\cl (G(Q_1)\setminus Q_1)=\bd _G (Q_1)$. By property (b) of $Q$ we have $\varphi (x)\in Q_2$. We infer that $x\in P_2$, which means that $P$ satisfies (b). By property (c) of $Q$ we have $\varphi (S)\subset\Int(Q_1\setminus Q_2)$. Thus, $S\subset\varphi ^{-1}(\Int(Q_1\setminus Q_2))\subset \Int(\varphi ^{-1}(Q_1)\setminus\varphi ^{-1}(Q_2))$, which proves property (c) of $P$. It remains to verify  (d). Let $x\in P_1\setminus P_2$. Then $\varphi(x)\in Q_1\setminus Q_2$ and property (d) of $Q$ yields $\varphi (x)\in \Int N$. Consequently, $x\in \varphi ^{-1}(\Int N)\subset \Int M$, and we have property (d) for $P$.  
 
By Lemma \ref{l1}(i) restrictions $F_{P,T_{M}(P)}(x):=F(x)$ and $G_{Q,T_N(Q)}(x):=G(x)$ are maps of pairs $F_{P,T_{M}(P)}:P\mto T_{M}(P)$ and $G_{Q,T_{N}(Q)}:Q\mto T_{N}(Q)$, respectively. Furthermore, we can treat the restriction of $\Psi$ as a map of pairs $\Psi_{Q,T_{M}(P)}:Q\mto T_{M}(P)$, because $\Psi (Q)=\Psi(\varphi (P))=F(P)\subset T_{M}(P)$. We have the following commutative diagram
\begin{equation}\label{d:d1}
\begin{tikzcd}
(P _1,P _2)\ar{d}[swap]{\varphi}\ar{r}{F} & (T_{M,1}(P),T_{M,2}(P)) \ar{d}{\varphi} & (P _1,P _2) \ar{l}[swap]{j_1}\ar{d}{\varphi} \\[3ex]
(Q_1,Q_2) \ar{r}[swap]{G}\ar{ur}{\Psi} & (T_{N,1}(Q),T_{N,2}(Q))& (Q_1,Q_2) \ar{l}{j_2}
\end{tikzcd}
\end{equation}
in which $j_1, j_2$ are inclusions. According to Lemma \ref{l1}(ii) inclusions $j_1,j_2$ induce isomorphisms in cohomology and we have well defined index maps $I_P:=H^*(F_{P,T_{M}(P)})\circ H^*(j_1)^{-1}$, $I_Q:=H^*(G_{Q,T_N(Q)})\circ H^*(j_2)^{-1}$ and $I_{QP}:=H^*(\Psi_{Q,T_{M}(P)})\circ H^*(j_1)^{-1}$. Consequently, diagram (\ref{d:d1}) results in the following commutative diagram in cohomology
$$
\begin{tikzcd}
H^*(P)& \ar{dl}{I_{QP}}\ar{l}[swap]{I_P}H^*(P)\\[3ex]
H^*(Q)\ar{u}{\varphi ^*}& \ar{l}{I_{Q}}H^*(Q)\ar{u}[swap]{\varphi ^*} 
\end{tikzcd}
$$
which means that $(H^*(P),I_P)$ and $(H^*(Q),I_Q)$ are linked in the sense of \cite{M90}; hence $L(H^*(P),I_P)$ and $L(H^*(Q),I_Q)$ are isomorphic.  
\qed%\end{proof}\\
%-------------------------------------

Now we provide examples showing that the conclusion of Theorem \ref{thm:com1} does not hold for an arbitrary continuous map $\varphi$.
\begin{ex}\label{ex:com1}
{\em Take $X:=[-5,5]$ and define $F:X\mto X$ by
\begin{equation}\label{eq:com1_F}
F(x):=\left\{
\begin{array}{rl}
\{-5\}&\for x\in [-5,-3)\\
{[-5,3]}&\for x=-3\\
{[-3,3]}&\for x\in (-3,-2]\\
{[2,3]}&\for x\in (-2,-1)\\
{[2,5]}&\for x=-1\\
\{5\}&\for x\in (-1,1)\\
{[2,5]}&\for x=1\\
{[2,3]}&\for x\in (1,2)\\
{[-3,3]}&\for x\in [2,3)\\
{[-5,3]}&\for x=3\\
\{-5\}&\for x\in (3,5].
\end{array}
\right.
\end{equation}
Consider $Y:=[0,5]$ and let $G:Y\mto Y$ be given by
\begin{equation}\label{eq:com1_G}
G(y):=|F(y)|\for y\in Y.
\end{equation}
The graphs of $F$ and $G$ are presented in Figure \ref{fig:com1} (A) and (B), respectively. 
%----------------------------------------------
\begin{figure}[h]
\begin{minipage}[t]{0.47\linewidth}
\centering
\includegraphics[width=\textwidth]{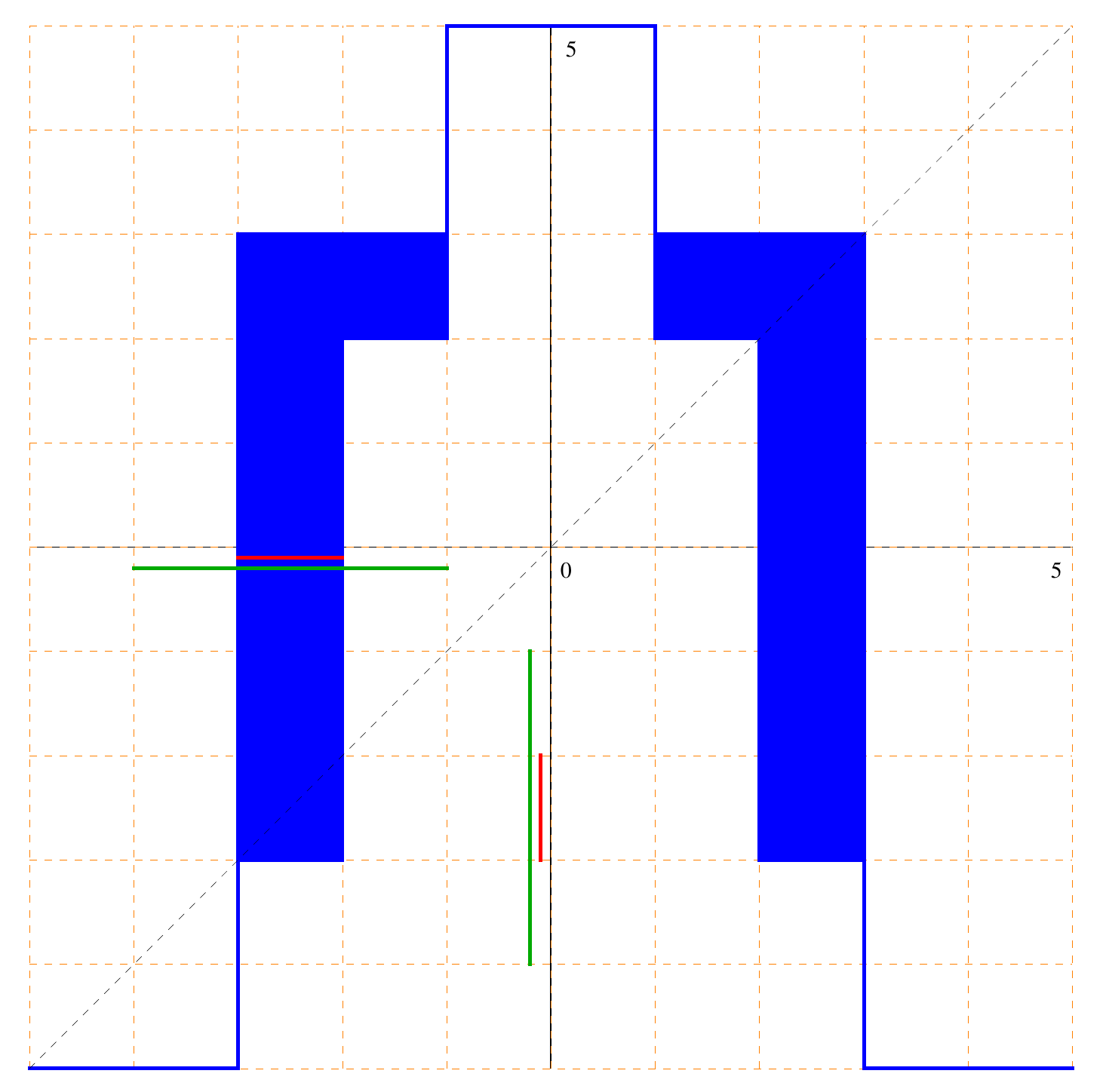}
\subcaption{The graph of $F$ given by (\ref{eq:com1_F}) is marked in blue. The isolated invariant set $S$ and its isolating neighborhood $M$ are indicated by line segments in red and green, respectively.}
\end{minipage}
\hspace{5mm}\begin{minipage}[t]{0.47\linewidth}
\centering
\includegraphics[width=\textwidth]{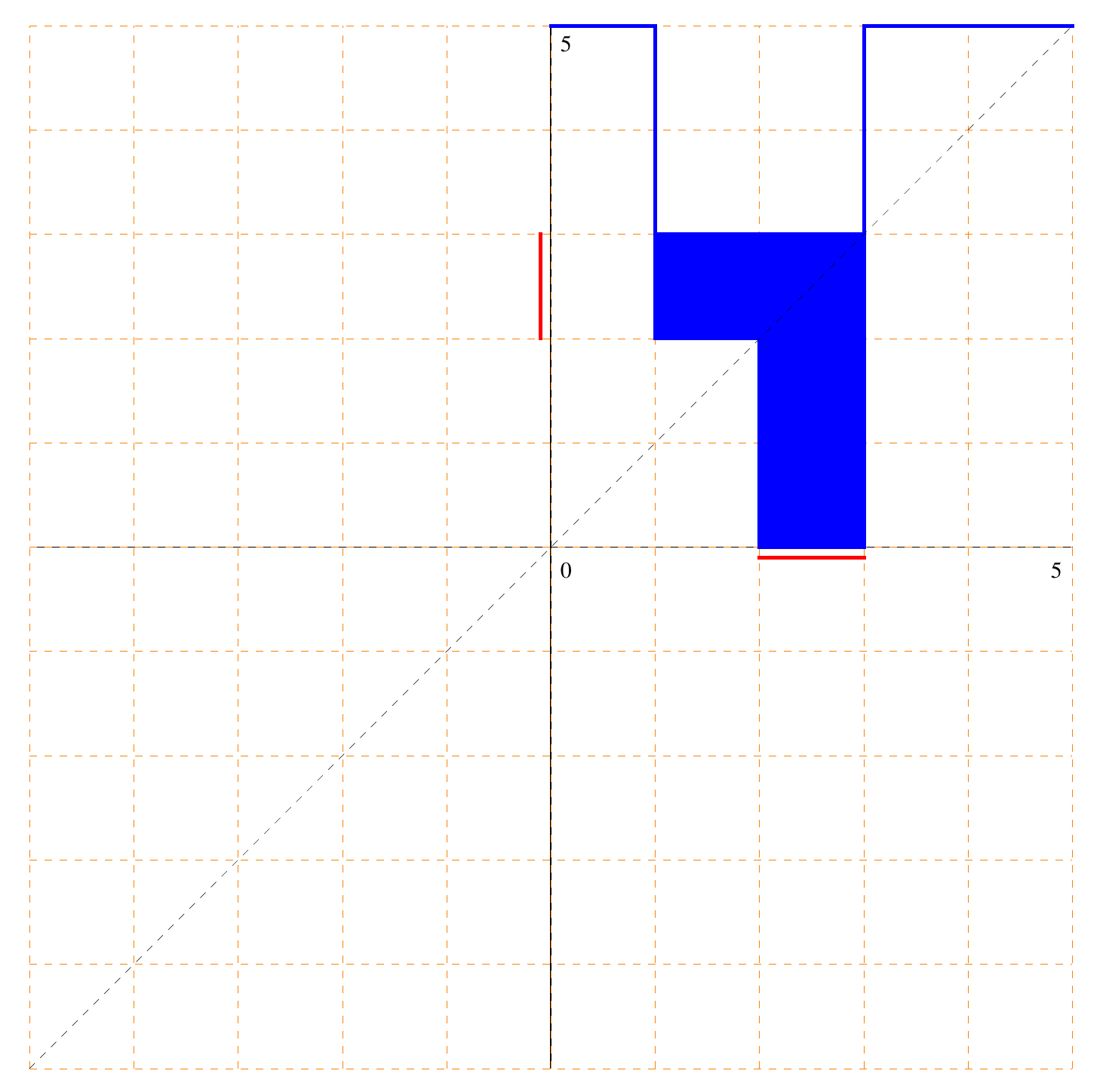}
\subcaption{The graph of $G$ given by (\ref{eq:com1_G}) is marked in blue. The image $\varphi(S)$ is indicated by a red line segment.}
\end{minipage}
\caption{Discrete multivalued dynamical systems $F$ and $G$, given by (\ref{eq:com1_F}) and (\ref{eq:com1_G}), respectively}\label{fig:com1}
\end{figure}
%--------------------------------
Take $\varphi :X\ni x\mapsto |x|\in Y$ and $\Psi:=F|_Y$, and observe that $\Psi\varphi=F$ and $\varphi\Psi=G$. Moreover, $S:=[-3,-2]$ is an isolated invariant set with respect to $F$, isolated, for instance, by $M:=[-4,-1]$, whereas $\varphi (S)=[2,3]$ is not an isolated invariant set with respect to $G$, as the invariant part  of any its compact neighborhood is significantly larger that $\varphi (S)$.   
}
\end{ex}
%-------------------------------------
\begin{ex}\label{ex:com2}
{\em We slightly modify the preceding example. Let $X:=[-5,5]$ and define $F:X\mto X$ by
\begin{equation}\label{eq:com_F}
F(x):=\left\{
\begin{array}{rl}
\{-5\}&\for x\in [-5,-3)\\
{[-5,3]}&\for x=-3\\
{[-3,3]}&\for x\in (-3,-2)\\
{[3,5]}&\for x=-2\\
\{5\}&\for x\in (-2,2)\\
{[-3,5]}&\for x=2\\
{[-3,3]}&\for x\in (2,3)\\
{[-5,3]}&\for x=3\\
\{-5\}&\for x\in (3,5]
\end{array}
\right.
\end{equation}
Again we consider $Y:=[0,5]$ and $G:Y\mto Y$  given by (\ref{eq:com1_G}).
%\begin{equation}\label{eq:com_G}
%G(y):=|F(y)|\for y\in Y.
%\end{equation}
The graphs of $F$ and $G$ are presented in Figure \ref{fig:com} (A) and (B), respectively. As before, $\varphi :X\ni x\mapsto |x|\in Y$ and $\Psi:=F|_Y$ satisfy $\Psi\varphi=F$ and $\varphi\Psi=G$. 

Consider $S:=[-3,-2]$, an isolated invariant set with respect to $F$, isolated by $M:=[-4,-1]$, and observe that the pair $P$, with $P_1:=M$ and $P_2:=\{-4\}\cup\{-1\}$ is a weak index pair with respect to $F$ in $M$. Then 
$$
H^k(P)=\left\{\begin{array}{rl}
\Z&\mbox{ for }k=1\\
0&\mbox{ for }k\neq 1
\end{array}
\right. 
$$
and the index map is the identity; hence
$$
C_k(S,F)=\left\{\begin{array}{rl}
(\Z,\id) &\mbox{ for }k=1\\
0&\mbox{ for }k\neq 1.
\end{array}
\right. 
$$
%----------------------------------------------
\begin{figure}[h]
\begin{minipage}[b]{0.47\linewidth}
\centering
\includegraphics[width=\textwidth]{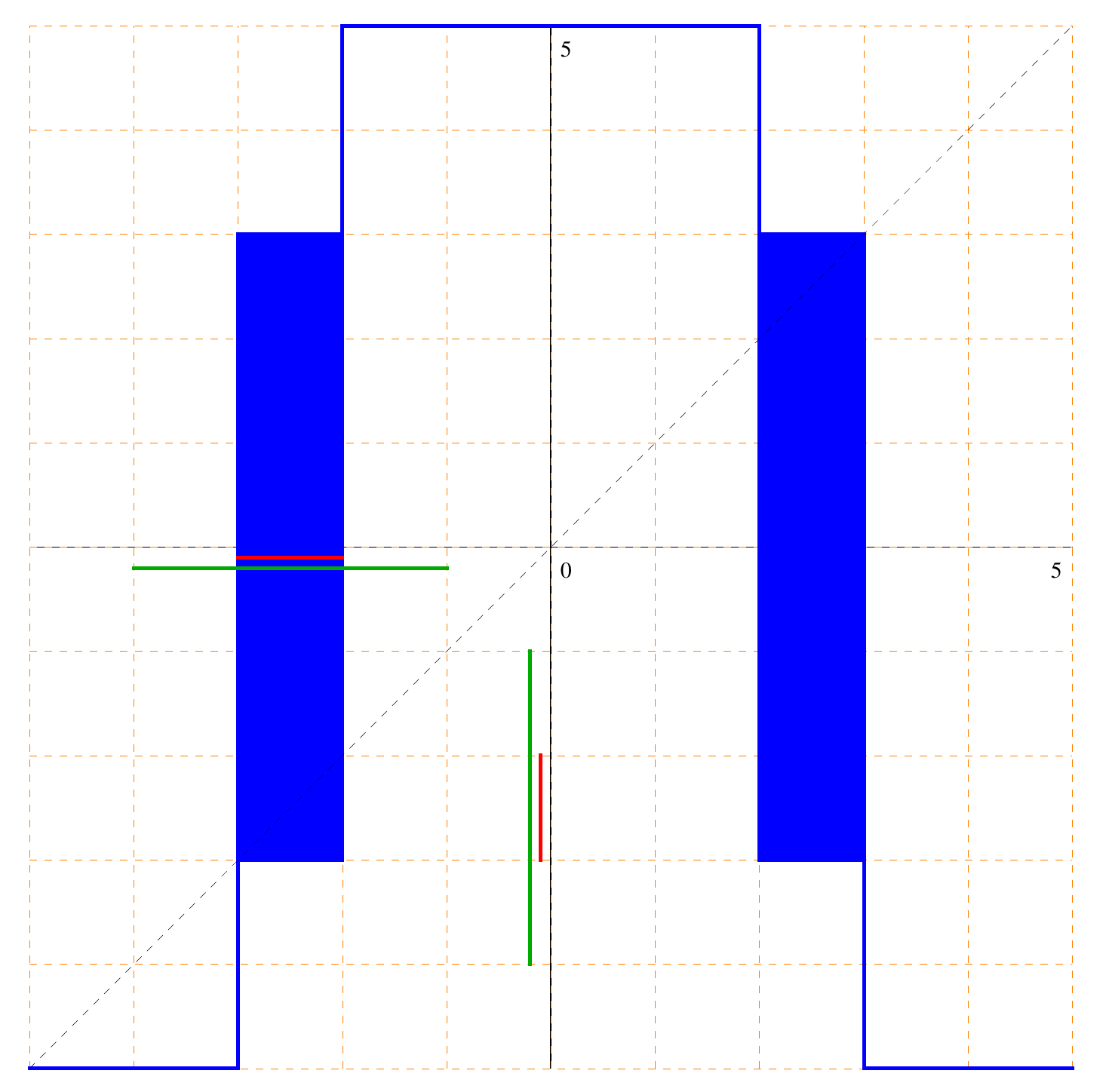}
\subcaption{The graph of $F$ given by (\ref{eq:com_F}) is marked in blue. The red and green line segments indicate the isolated invariant set $S$ and its isolating neighborhood $M$, respectively.}
\end{minipage}
\hspace{5mm}\begin{minipage}[b]{0.47\linewidth}
\centering
\includegraphics[width=\textwidth]{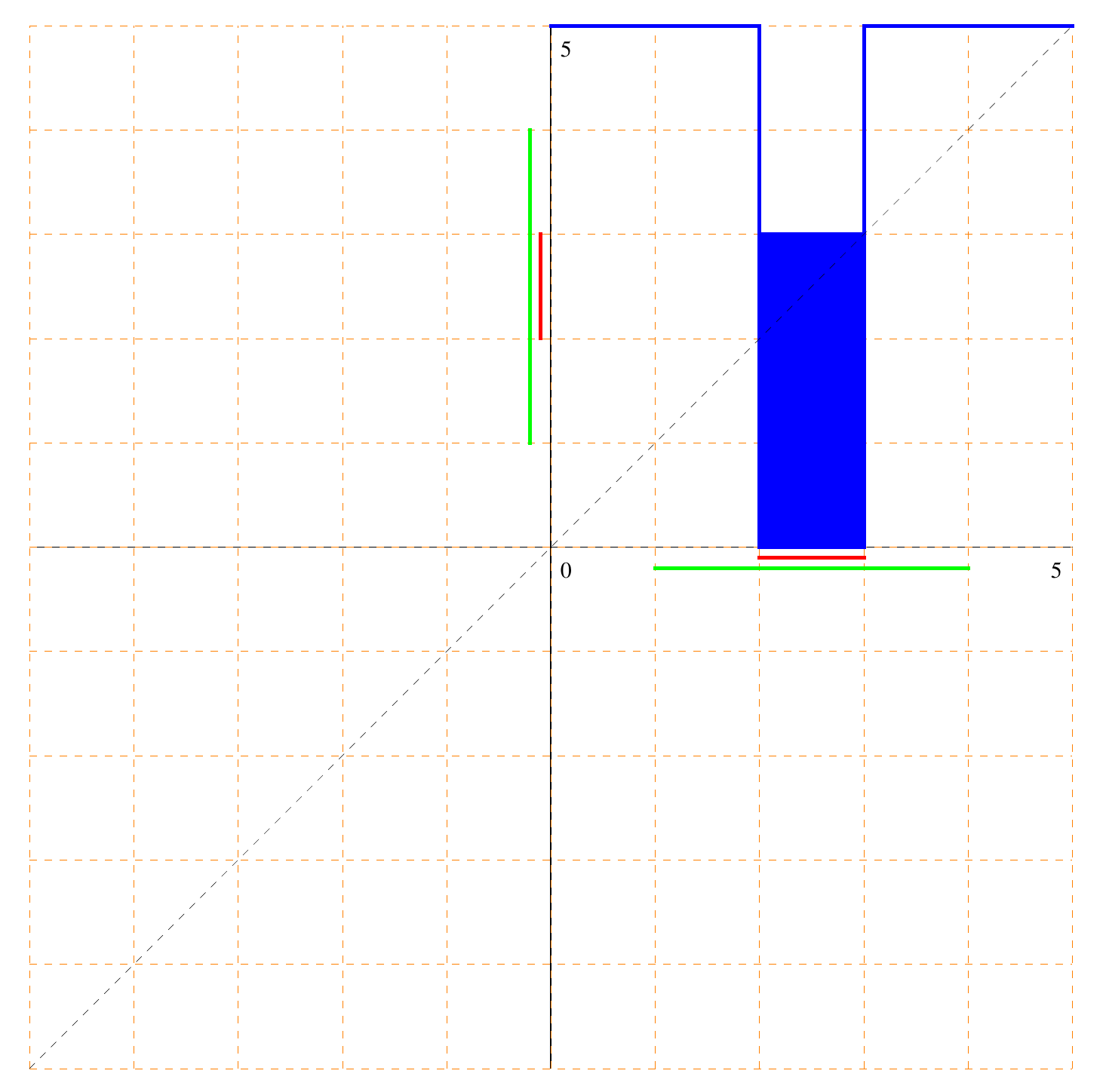}
\subcaption{The graph of $G$ given by (\ref{eq:com1_G}) is marked in blue. The red and green line segments indicate the the isolated invariant set $\varphi(S)$ and its isolating neighborhood $N$, respectively.}
\end{minipage}
\caption{Isolated invariant sets with respect to discrete multivalued dynamical systems $F$ and $G$}\label{fig:com}
\end{figure}
%--------------------------------

Now we take under consideration $\varphi (S)=[2,3]$, which is an isolated invariant set with respect to $G$. One easily verifies that $N:=[1,4]$ isolates $\varphi(S)$ and $Q=(Q_1,Q_2)$, where $Q_1:=N$, $Q_2:=\{1\}\cup\{4\}$, is a weak index pair in $N$. We have 
$$
H^k(Q)=\left\{\begin{array}{rl}
\Z &\mbox{ for }k=1\\
0&\mbox{ for }k\neq 1
\end{array}
\right. 
$$
and $I_Q=0$. Thus, $C(\varphi(S),G)$ is trivial, showing that
$C(S,F)\neq C(\varphi(S),G)$.
}
\end{ex}
%-------------------------------------
Similarly as in the single valued case, we are able to apply Theorem \ref{thm:com1} for an inclusion in order to treat the restriction of a given discrete multivalued dynamical system to an invariant subspace.
\begin{thm}
Let $A\subset X$ be a locally compact subset of $X$ such that $F(X)\subset A$. If $S$ is an isolated invariant set with respect to $F$ then $S$ is an isolated invariant set with respect to $F_{|A}$ and $C(S,F)=C(S,F_{|A})$.
\end{thm}
\begin{proof}
For the proof we apply Theorem \ref{thm:com1} with the inclusion $i_A:A\to X$ and $F$ in the role of $\varphi$ and $\Psi$, respectively.
\qed\end{proof}
%-----------------------------------------
If an isolated invariant set $S$ is actually strongly isolated, then the conclusion of Theorem \ref{thm:com1} holds for any continuous $\varphi$, not necessarily injective.
\begin{thm}\label{thm:com2}
Let $F:X\mto X$ and $G:Y\mto Y$ be given discrete multivalued dynamical systems and let 
$\varphi:X\to Y$ and $\Psi:Y\mto X$ be partial maps such that $\varphi$ is continuous, $\dom\varphi$ is compact, $\Psi$ is upper semicontinuous, $F=\Psi\varphi$ and $G=\varphi\Psi$. Assume that $S\subset X$ is a strongly isolated invariant set with respect to $F$, with some its strongly isolating neighborhood $M$ contained in the domain of $\varphi$, and the domain of $\Psi$ contains a neighborhood of $\varphi (S)$. Then $\varphi (S)$ is an isolated invariant set with respect to $G$, and $C(S,F)=C(\varphi (S),G)$. 
\end{thm}
\begin{proof}
Since $M$ is a strongly isolating neighborhood of $S$, we have $F(S)=\Psi(\varphi (S))\subset\Int M$. By the upper semicontinuity of $\Psi$, there exists an open neighborhood $V$ of $\varphi(S)$ such that $\Psi(V)\subset \Int M$. Thus, we can take $N$, a compact neighborhood of $\varphi (S)$ such that 
\begin{equation}\label{eq:com2e}
\Psi (N)\subset \Int M.
\end{equation}
Without loss of generality we can assume that $N\subset\dom\Psi$. 

We will show that $N$ is an isolating neighborhood of $\varphi (S)$. Take $y\in \Inv _GN$ and consider a solution $\tau$ with respect to $G$ in $N$ through $y$, i.e. $\tau:\Z\to N$ with $\tau (0)=y$ and $\tau (k)\in G(\tau(k-1))$ for $k\in \Z$. We have $\tau (k)\in\varphi\Psi(\tau (k-1))\cap N$, hence we can take $x\in\Psi(\tau (k-1))$ such that $\tau (k)=\varphi (x)$ and define $\sigma$ by setting $\sigma (k):=x$, for any $k\in\Z$. Then we have $\sigma (k)\in\Psi (\tau(k-1))=\Psi (\varphi (\sigma (k-1)))=F(\sigma (k-1))$. This along with (\ref{eq:com2e}) means that we have defined a solution with respect to $F$ in $M$. Thus $\sigma (\Z)\subset S$, as $M$ is an isolating neighborhood for $S$. Consequently, $\tau (\Z)\subset \varphi (\sigma (\Z))\subset \varphi (S)$. In particular, $y\in\varphi (S)$, which shows that $\Inv _GN\subset\varphi (S)$. For the proof of the opposite inclusion consider $y\in \varphi (S)$ and $x\in S$ with $y=\varphi (x)$. Let $\sigma $ be a solution with respect to $F$ through $x$ in $S$, i.e. $\sigma:\Z\to S$, $\sigma (0)=x$ and $\sigma (k)\in F(\sigma (k-1))$ for $k\in\Z$. Define $\tau :=\varphi \circ \sigma$. One can verify that $\tau$ is a solution with respect to $G$ in $\varphi (S)$ through $y$, which yields $\varphi(S)\subset \Inv _GN$.

Now we observe that $\varphi ^{-1}(N)$ is an isolating neighborhood of $S$ with respect to $F$. We have $\Inv _F(S)=S$ and $S\subset \Int \varphi ^{-1}(N)$, therefore, it suffices to show that $\Inv _F\varphi ^{-1}(N)\subset S$. Take $x\in \Inv _F\varphi ^{-1}(N)$ and consider a solution $\sigma $ with respect to $F$ in $\varphi ^{-1}(N)$ through $x$. 
By (\ref{eq:com2e}), for any $k\in\Z$, we have $\sigma(k)\in F(\sigma (k-1))=\Psi(\varphi (\sigma(k-1))\subset M$. Thus, $\sigma$ is a solution with respect to $F$ in $M$. But $M$ is an isolating neighborhood of $S$, hence $\sigma (\Z)\subset S$ follows. In particular $x\in S$.

Let $Q=(Q_1,Q_2)$ be a weak index pair in $N$. Define $P_i:=\varphi ^{-1}(Q_i)$, $i=1,2$. We show that $P:=(P_1,P_2)$ is a weak index pair in $\varphi ^{-1}(N)$. Clearly $P_1$ and $P_2$ are compact, and $P_2\subset P_1\subset\varphi ^{-1}(N)$. Using the same reasoning as in the proof of Theorem \ref{thm:com1} one can show that $P$ satisfies conditions (a), (c) and (d). It remains to verify (b). To this end consider $x\in\bd _F (P_1)$. Since $x\in\cl(F(P_1)\setminus P_1)$, we can take a sequence $\{x_n\}\subset F(P_1)\setminus P_1$ convergent to $x$. Clearly $\varphi (x_n)$ converges to $\varphi (x)$. Moreover, for any $n$ we have $\varphi (x_n)\in \varphi (F(P_1))=G(Q_1)$ and $\varphi (x_n)\notin Q_1$. Thus, we have $\varphi (x_n)\in G(Q_1)\setminus Q_1$ and, consequently, $\varphi (x)\in \cl (G(Q_1)\setminus Q_1)$. We also have $\varphi (x)\in Q_1$, as $x\in\bd_F(P_1)\subset P_1$. This means that $\varphi (x)\in \bd_G(Q_1)$ which, according to property (b) of $Q$, implies $\varphi (x)\in Q_2$ and $x\in P_2$.

The remaining part of the proof runs along the lines of an appropriate part of the proof of Theorem \ref{thm:com1}. 
\qed\end{proof}
%%%%%%%%%%%%%%%%%%%%%%%%%%%%%%%%%%%%%%%%%%%%%%5
\section*{Acknowledgements}
I would like to express my sincere gratitude to Professor Marian Mrozek 
for encouraging me to undertake the present research, and for numerous 
inspiring discussions.

%%%%%%%%%%%%%%%%%%%%%%%%%%%

\end{document}